\documentclass[aap,preprint]{imsart}

\RequirePackage{amsthm,amsmath,amsfonts,amssymb}
\RequirePackage[numbers,sort&compress]{natbib}
\RequirePackage[colorlinks,citecolor=blue,urlcolor=blue]{hyperref}
\RequirePackage{graphicx}
\usepackage{dsfont}
\usepackage{stmaryrd}
\usepackage{amsmath}
\usepackage{mathtools}
\usepackage[autostyle]{csquotes}
\usepackage{lineno}
\usepackage{float}

\newcommand{\g}{\gamma}
\renewcommand{\a}{\alpha}

\renewcommand{\d}{\delta}
\newcommand{\s}{\sigma}

\newcommand{\fd}{\textnormal{f.d.}}
\newcommand{\sgn}{\textnormal{sgn}}
\newcommand{\E}{\mathbb{E}}

\newcommand{\Q}{\mathbb{Q}}
\newcommand{\R}{\mathbb{R}}
\newcommand{\N}{\mathbb{N}}
\newcommand{\Z}{\mathbb{Z}}

\newcommand{\C}{\mathbb{C}}

\newcommand{\A}{\mathcal{A}}

\newcommand{\FF}{\mathcal{F}}
\newcommand{\CC}{\mathcal{C}}

\newcommand{\GG}{\mathcal{G}}

\newcommand{\PP}{\mathcal{P}}

\newcommand{\MM}{\mathcal{M}}
\newcommand{\NN}{\mathcal{N}}

\renewcommand{\SS}{\mathcal{S}}
\newcommand{\ee}{\mathbb{\varepsilon}}
\newcommand{\1}{\mathds{1}}

\newcommand{\ie}{\textnormal{i.e.}}

\newcommand{\Marg}{\textnormal{Marg}}
\newcommand{\Law}{\textnormal{Law}}

\newcommand{\ti}[1]{\tilde{#1}}

\newcommand{\Tr}{\textnormal{Tr}}
\newcommand{\proj}{\textnormal{proj}}
\newcommand{\GL}{\textnormal{GL}}
\newcommand{\n}[2]{\left\|#1\right\|_{#2}}

\newcommand{\dcv}[1]{\xrightarrow[#1]{}}

\newcommand{\accol}[1]{\left\{\begin{array}{l}#1\end{array}\right .}
\newcommand{\ens}[2]{\left\{ #1 ~;~#2 \right\} }

\startlocaldefs
\theoremstyle{plain}

\newtheorem{them}{Theorem}
\newtheorem{pro}[them]{Proposition}
\newtheorem{cor}[them]{Corollary}
\newtheorem{lem}[them]{Lemma}
\newtheorem{theoremA}{Theorem}%
\newtheorem{proA}[theoremA]{Proposition}

\newtheorem{remq}[them]{Remark}
\newtheorem{ex}[them]{Example}
\newtheorem{nott}[them]{Notation}
\newtheorem{notdef}[them]{Notation/Definition}
\newtheorem{defi}[them]{Definition}

\numberwithin{them}{section}

\theoremstyle{definition}

\endlocaldefs

\begin{document}

\begin{frontmatter}
\title{On the Markov transformation of Gaussian processes}
\runtitle{On the Markov transformation of Gaussian processes}

\begin{aug}
\author[A]{\fnms{Armand}~\snm{Ley} \ead[label=e1]{armand.ley@univie.ac.at}}
\address[A]{Department of Mathematics,
University of Vienna\printead[presep={ ,\ }]{e1}}
\end{aug}

\begin{abstract}
Starting from a process $(X_t)_{t \in \R}$, we study a procedure introduced by Boubel and Juillet called the \emph{Markov transformation}, which aims to construct a process that evolves like $(X_t)_{t \in \R}$ while forgetting its past. When $(X_t)_{t \in \R}$ is a sufficiently regular Gaussian process, we prove that its Markov transform is well defined and uniquely characterized: it is the only Gaussian Markov process with the same one-dimensional marginals and instantaneous decorrelation rate. Answering a question of Boubel and Juillet concerning the uniqueness in law of Markov transforms, we construct a highly irregular Gaussian process that admits infinitely many Markov transforms.
\end{abstract}



\end{frontmatter}

\section{Introduction}\label{sec:Intro}

\subsection{Mimicking problems}

In recent decades, partly motivated by mathematical finance, the problem of mimicking stochastic processes has gained increasing attention. In general terms, the problem can be expressed as follows: Given a stochastic process $(X_t)_{t \in T}$, is there  a process $(Y_t)_{t \in T}$ preserving certain properties of $(X_t)_{t \in T}$ while satisfying some additional conditions? To motivate our problem of mimicking processes, we now present a selection of five mimicking problems: 
\begin{longlist}
	\item The Kellerer problem of mimicking a martingale by a \emph{Markov} martingale;
	\item The Gyöngy problem of mimicking an Itô process by the solution of a SDE;
	\item The problem of faking Brownian motion;
	\item The problem of mimicking an $\R^2$-valued process by an $\R^2$-valued order-preserving Markov process;
	\item The problem of Boubel and Juillet on mimicking an increasing process by an \emph{Markov} increasing process.
\end{longlist}
The last mimicking problem is central to our work, as the study of Markov transformation, the method used to construct its solution, is the main objective of this article.

In a seminal article, Strassen \cite{strassen_existence_1965} examined whether, given a coupling $(X_t)_{t \in \{1,2\}}$, there exists a coupling $(Y_t)_{t \in \{1,2\}}$ with the same $1$-marginals that satisfies the martingale property, \ie , \ $\E(Y_2|Y_1) = Y_1$. He proved that such a coupling exists if and only if $X_1$ is smaller than $X_2$ in the convex order, \ie ,\ $\E(f(X_1)) \leq \E(f(X_2))$ for every function $f$ that is convex\footnote{We refer to \cite{muller_comparison_2002,shaked_stochastic_2007} for more information about stochastic orders.}. Kellerer \cite{kellerer_markov-komposition_1972} generalized this result: A real-valued process $(X_t)_{t \in \R}$ can be mimicked by a \emph{Markov} martingale having the same $1$-marginals if, and only if, its components are increasing in the convex order (see also \cite{beiglbock_root_2016,boubel_markov-quantile_2022,hirsch_kellerers_2015}). In a different vein Gyöngy \cite{gyongy_mimicking_1986}, using an approach suggested by Krylov \cite{krylov_once_1985}, showed that any Itô process $(X_t)_{t \in \R}$ with coefficients satisfying certain conditions can be mimicked by a Markov process which is a solution of a stochastic differential equation and which has the same one-dimensional marginals. Another mimicking problem is to fake the Brownian motion, that is, to find a non-Brownian process that has as many common features with the Brownian motion as possible. It is known (see e.g. \cite{lowther_fitting_2008}) that the Brownian motion is the only continuous martingale with Brownian marginals that satisfies the \emph{strong} Markov property. Beiglböck, Lowther, Pammer and Schachermayer \cite{beiglbock_faking_2024}, in the line of previous articles (see their introduction for numerous references) finally proved that there exists a non-Brownian continuous martingale with Brownian marginals that is Markov. More recently, for processes valued in product spaces, Bérard and Frénais \cite{berard_comonotone_2024}  studied the case of a homogeneous Markov process $(X^{x_1}_t, X^{x_2}_t)_{t \geq 0}$ starting at $x_1 \leq x_2$ and whose marginal processes $ (X^{x_i}_t)_{t \geq 0}$ are governed by the same stochastically monotone Feller semi-group. They showed that it can be mimicked by a Feller process $(Y^{x_1}_t, Y^{x_2}_t)_{t \geq 0}$ starting at $(x_1,x_2)$, satisfying $\text{Law}\left((X^{x_i}_t)_{t \geq 0} \right) = \text{Law}\left((Y^{x_i}_t)_{t \geq 0}\right)$ for $i = 1,2$ and $Y_t^{x_1} \leq Y_t^{x_2}$ for every $t \in \R_+.$ 

Directly related to our work, Boubel and Juillet \cite{boubel_markov-quantile_2022} studied the problem of mimicking a stochastic process  $(X_t)_{t \in \R}$ by a \emph{Markov} process with the same $1$-marginals and with non-decreasing trajectories. If we ignore the Markov property, it is well established that there exists a solution if, and only if, the marginals are increasing in the stochastic order, \ie , $\E(f(X_s)) \leq \E(f(X_t))$ for every $s<t \in \R^2$ and $f : \R \to \R$ non-decreasing. An explicit solution is then given by the quantile process, \ie , \ by $(G_t(U))_{t \in \R}$ where $U$ is a uniformly distributed random variable on $[0,1]$ and  $G_t : q \in ]0,1[ \mapsto \inf \big( \ens{x \in \R}{\mathbb{P}(X_t \leq x) \geq q} \big) \in \R$ stands for the quantile function of $X_t.$ However, as shown in \cite{juillet_peacocks_2016}, the quantile process is not Markov in general. To overcome this problem, Boubel and Juillet introduced a procedure, which we will call here Markov transformation\footnote{In \cite[Theorem 2.26]{boubel_markov-quantile_2022}, the authors refer to this operation as \emph{Markovinification}.}, that will be the object of study of this article. Applied to the quantile process, this procedure leads to the so-called Markov-quantile process, which is a solution of the previous mimicking problem.

Before giving more details about the Markov transformation and its application to mimicking, we emphasize that mimicking constructions can be roughly divided into two classes. The first relies on strong structural assumptions on the underlying process, typically on properties of its sample paths, whereas the second operates at the level of finite-dimensional distributions. A prominent example of the first class is the construction of Gyöngy, where the underlying process is an Itô process with uniformly elliptic coefficients. An example of the second class is Kellerer’s construction. As we will see below, the Markov transformation introduced by Boubel and Juillet clearly belongs to this second class. Consequently, it can naturally be applied to highly irregular processes --such as the quantile processes considered in \cite{boubel_markov-quantile_2022}-- where other mimicking strategies may not even be defined.

\subsection{Markov transformation}
	At an intuitive level the Markov transform of a process $X$ is the process obtained when one evolves following the law of $X$ while instantaneously forgetting the past of the trajectory. As we will see, making this rigorous is delicate and will be done --when possible-- through a limiting procedure. Given a finite set of times $R = \{r_1, \dots ,r_m\}$, we start by introducing the process $X$ ``made Markov at times $R$'', that is the process that evolves following the law of $X$, but forgets about the past of the trajectory at times $r_i$, $i=1,\dots,m$. The Markov transform will then be defined as limit of such processes when the set of times $R$ are expanding with vanishing mesh. We provide here an informal definition of the process ``made Markov at times $R$''. For the formal definition, see Definition~\ref{def:made_markov}.
\begin{defi}
	Given a finite set of times $R \subset \R$ and a process $X = (X_t)_{t \in \R}$, we say that a process $X^R$ is the\footnote{ Since all processes made Markov at times $R$ share the same distribution, we refer to \emph{the} process $X$ made Markov at times $R$.} transformation of $X$ made Markov at times $R$ if: 
	\begin{itemize}
		\item On every interval between two successive times of $R$, $X^R$ and $X$ follow the same distribution;
		\item For every $r \in R$, $X^R$ is made Markov at time $r$: Knowing the value of the trajectory at time $r$ implies that the future trajectory $(X_t^R)_{t>r}$ does not depend on the past trajectory $(X_t^R)_{t<r}$.
	\end{itemize}
\end{defi}
A Markov transform of $X$ is now defined as a process obtained by taking a family of expanding sets of times with vanishing mesh. More precisely, we say that a sequence $(R_n)_{n \geq 1}$ of subsets of $\mathbb{R}$ is an \emph{admissible sequence of sets of times} if
\[
\lim_{n \to +\infty} \inf(R_n) = -\infty, 
\qquad 
\lim_{n \to +\infty} \sup(R_n) = +\infty,
\]
and
\[
\lim_{n \to +\infty} \sup_{x \in R_n} d\bigl(x, R_n \setminus \{x\}\bigr) = 0.
\]

\begin{defi}
	A Markov process $X'$ is said to be a \emph{Markov transform of $X$} if it is the limit, with respect to the finite-dimensional topology, of a sequence of processes $\bigl(X^{R_n}\bigr)_{n \geq 1}$, where $(R_n)_{n \geq 1}$ is an admissible sequence. 
\end{defi}

As mentioned above, an important motivation behind the Markov transformation is that it may provide a general tool for establishing \emph{Markov mimicking} results. First, observe that \(X^R\) has the same one-dimensional marginals as \(X\); consequently, any Markov transform of \(X\) shares these marginals as well. At an intuitive level, enforcing the Markov property on a family of expanding sets of times with vanishing mesh and then passing to the limit should produce a Markov process. The underlying hope is that, given a property of interest \(\mathcal{P}\) satisfied by the initial process \(X\), enforcing the Markov property on progressively finer sets of times and then passing to the limit may preserve \(\mathcal{P}\). If so and Markov transformation indeed carries out, the resulting Markov transform would mimic $X$, in the sense that it would preserve the one-dimensional marginals and the property \(\mathcal{P}\), while additionally enjoying the Markov property. We distinguish two contexts when studying the Markov transformation, depending on whether the underlying process satisfies certain regularity properties or not.

In the context of a regular underlying process, a given mimicking problem typically admits an independent construction of its solution. Recovering this solution through the Markov transformation offers a new perspective, as it allows one to view it as the process obtained by following the law of the original process while systematically forgetting its past. This is, in itself, a noteworthy result. The \emph{first main contribution} of this paper follows this idea. We first prove in Proposition \ref{proA:mimicking} that, under mild regularity assumptions --namely, the continuity of the instantaneous decorrelation rate-- a particular mimicking problem admits a unique solution, which can be characterized as the solution of a suitably chosen stochastic differential equation. In Theorem \ref{them:transform_is_mimicking}, we recover this solution as the Markov transform of the initial process. This provides the first non-trivial yet tractable class of processes for which this abstract Markov transform procedure allows one to reinterpret the solution of a mimicking problem obtained by independent mean.

In the context of an irregular underlying process, the hope is that the Markov transformation provides a generic mimicking tool when no other construction is available. This program was carried out by Boubel and Juillet \cite[Theorem B]{boubel_markov-quantile_2022} for quantile processes, which may be highly irregular. The \emph{second main contribution} of our paper is to deepen the understanding of the Markov transformation when the underlying process does not satisfy any standard regularity assumptions. First, although processes made Markov at a fixed set of times $R$ can always be constructed, the limiting step of the procedure is by no means guaranteed. Several natural questions arise, for instance concerning existence, uniqueness, or the obtention of the Markov property in the limit. In the Gaussian setting, we use an extension theorem of Kellerer to prove that any Gaussian process --even with a very irregular covariance structure-- admits a (local) Markov transform. This result, proved in Proposition~\ref{proA:existence}, complements the earlier result of Boubel and Juillet for processes with increasing kernels and shows that the Markov transformation can be carried out in a wide range of contexts. Next, we answer an open question raised in \cite[$\mathsection 5.5.1$, Open Question~(a)]{boubel_markov-quantile_2022} concerning the uniqueness of a Markov transform. Given two admissible sequences $(R_n^1)_{n \geq 1}$ and $(R_n^2)_{n \geq 1}$ such that $(X^{R_n^1})_{n \geq 1}$ and $(X^{R_n^2})_{n \geq 1}$ converge to Markov processes $X^{(1)}$ and $X^{(2)}$, must these limiting processes coincide in law? We answer this question negatively in Theorem~\ref{them:counterexample_intro} by constructing a highly irregular Gaussian process that admits infinitely many distinct Markov transforms.

Finally, we stress that the general question of the existence of a Markov transform remains open. It is not known whether, given a process $X$, there always exists an admissible sequence $(R_n)_{n \geq 1}$ and a Markov process $X'$ such that $(X^{R_n})_{n \geq 1}$ converges weakly to $X'$. Likewise, the Markovian character of a potential limit is not yet fully understood: given a process $X$ and an admissible sequence $(R_n)_{n \geq 1}$ such that $X^{R_n}$ converges to a process $X'$, is $X'$ necessarily Markov? We now state in details the above mentioned results.

\subsection{Main results}
  The first main result establishes that the Markov transform procedures carries out for regular Gaussian processes. Indeed, for such process, there exists a unique Markov transform, and it admits a closed-form expression. Furthermore, it can be characterized as the solution to the following suitable mimicking problem. 
  
  Let $X = (X_t)_{t \in \mathbb{R}}$ be a Gaussian process with covariance function
  \[
  K : (s,t) \longmapsto \mathbb{E}(X_s X_t) - \mathbb{E}(X_s)\mathbb{E}(X_t),
  \]
  and assume that, for every $t \in \mathbb{R}$, the decay rate of the correlation of $X$ at time $t$, defined by
  \[
  L_t : h \in ]0,+\infty[ \longmapsto \frac{1}{h}\left(1 - \frac{K(t,t+h)}{\sqrt{K(t,t)}\,\sqrt{K(t+h,t+h)}} \right),
  \]
  converges in $[0,+\infty]$ as $h \to 0^+$. This limit, called the \emph{instantaneous decorrelation rate} (or instantaneous decay rate of the correlation) of $X$ at time $t$, is denoted by $\alpha^K(t)$ so that
  \begin{equation}\label{eq:IDR}
  	\alpha^K(t)
  	= \lim_{h \to 0^+} \frac{1}{h}\left(1 - \frac{K(t,t+h)}{\sqrt{K(t,t)}\,\sqrt{K(t+h,t+h)}} \right) \in [0,+\infty].
  \end{equation}
  
  The mimicking problem associated with Markov transformation of Gaussian processes then consists in finding a Gaussian \emph{Markov} process that shares the same one-dimensional marginals as $X$ \emph{and} has the same instantaneous decorrelation rate. The following result shows that when the instantaneous decorrelation rate is continuous, this mimicking problem admits a unique solution, that is completely explicit. Under suitable regularity assumptions on the mean and the variance function, it can furthermore be characterized as the solution to a stochastic differential equation.

\begin{proA}\label{proA:mimicking}
	Let  $X=(X_t)_{t \in \R}$ be a Gaussian process with  continuous covariance function $K$ and positive variance function. Assume that for all $t \in \R$, $L_t(\cdot)$ admits a limit at $0^+$ and that $\alpha^K$, the instantaneous decay rate of $K$ defined in \eqref{eq:IDR},  is finite and continuous. 		
	\begin{enumerate}
		\item\label{Existence} \emph{Existence:} There exists a process $Y =(Y_t)_{t \in \R}$ with covariance function $K'$ satisfying:
		\begin{enumerate}
			\leftskip=0.3in
			\item[(a)]\label{point:same_marg} For all $t \in \R$, $\text{Law}(X_t) = \text{Law}(Y_t)$;
			\item[(b)]\label{point:same_deco_rate} $\a^{K'} = \a^{K}$;
			\item[(c)]\label{point:gaussian} $(Y_t)_{t \in \R}$ is a Gaussian Markov process.
		\end{enumerate}
		Moreover, for all $s<t \in \R^2$,
		\begin{equation}\label{eq:univ_conv}
			\sup_{v \in  [s,t] } \left| \a^K(v) - L_v^{K'}(h) \right| \dcv{h \to 0^+} 0.
		\end{equation} 
		If a process $Y$ satisfies (a), (b) and (c), we say that $Y$ is a \emph{mimicking process} of $X.$
		\item\label{Point:Uniqueness} \emph{Uniqueness in law:} Every mimicking process of $X$ with covariance function $K' :\R^2 \to \R$ has the same mean function as $(X_t)_{t \in \R}$, and for all $s<t \in \R^2$,
		\begin{equation}\label{eq:transfor_formule_cov}
			K'(s,t) = \sqrt{K(s,s)K(t,t)} \exp\left(-\int_s^t \a^K(u) du\right).
		\end{equation}
		\item\label{Point:SDE} \emph{Underlying dynamic of the mimicking process: } Assume $m : t \mapsto \E(X_t)$ and $\sigma : t \mapsto K(t,t)^{1/2}$ are continuously differentiable. Then strong existence and strong uniqueness\footnote{Strong uniqueness is meant in the sense of \cite[Chapter $5$, Definition $2.3$]{karatzas_brownian_1988}. By strong existence, we mean existence of a strong solution for any given Brownian motion and independent initial condition. We refer to \cite[Chapter $5$, Definition $2.1$]{karatzas_brownian_1988} for the definition of a strong solution.} hold for the  SDE 
		\begin{equation}
			\begin{cases*}\label{eq:SDE_complet}
				dZ_t = \left(m'(t) + \left(\frac{\sigma'(t)}{\sigma(t)} - \alpha^K(t)\right)\bigl(Z_t - m(t)\bigr)\right)dt
				+ \sigma(t)\sqrt{2\alpha^K(t)}\,dB_t \\
				Z_0 \sim \text{Law}(X_0)
			\end{cases*}.
		\end{equation}
		Moreover, the\footnote{Since strong uniqueness holds, this process is unique up to indistinguishability.} law of its solution is the law of the mimicking process of $(X_t)_{t \geq 0}$.
	\end{enumerate}
\end{proA}

\begin{remq}\label{them:time_changed_OU}
	Assume $(Z_t)_{t \in \R}$ is a standard stationary Ornstein-Uhlenbeck process, meaning that $dZ_t = - Z_t dt + d B_t$, where $Z_0 \sim \NN(0,1/2)$ is independent from $(B_t)_{t \geq 0}$. Then the mimicking process coincides with the time changed and rescaled process $Y = (Y_t)_{t  \geq 0}$ given by
	\begin{equation*}
		Y_t =  m(t) + \sqrt{ 2K(t,t)}Z_{\Phi(t)},
	\end{equation*}
	where $\Phi(t) := \int_0^t \alpha^K(u)du$ for all $t \geq 0$. Indeed, $Y$ has the same mean function than $X$ and, as $\left(Z_{\Phi(t)}\right)_{t\in \R}$ is a centred Gaussian process such that, 
	$$\E(Z_{\Phi(s)}Z_{\Phi(t)}) = \frac12 \exp\left(-|\Phi(t)-\Phi(s)|\right) = \frac12\exp\left( -\int_s^t \a^K(u) du \right) \quad \text{for all } s<t,$$ the covariance of $Y$ is given by Equation \eqref{eq:transfor_formule_cov}.
\end{remq}
Note that without requirement (b), the mimicking problem becomes trivial and its solution highly non-unique: for instance a continuum of independent random variables with the same one-dimensional marginals satisfies (a) and (c). At this stage one should think of the instantaneous decorrelation as the right measure of regularity for Gaussian processes in the context of Markov transformation. This will be made clearer by a formal computation after the statement of Theorem \ref{them:transform_is_mimicking}. 

Observe that for any $s<t \in \R^2$, the correlation of $(Y_s,Y_t)$ is inversely proportional to the exponential value of the integral of the instantaneous decorrelation rates between $s$ and $t.$ Hence, informally, if $X$ is highly correlated between $s$ and $t$, the instantaneous decorrelation rate of $X$ will be small between $s$ and $t$, so that the correlation coefficient of $(Y_s,Y_t)$ will be high.

\begin{ex}[Fractional Brownian motion]\label{ex:fbm}
	Consider $H \in ]0,1[$ and let $(X_t^H)$ denote the fractional Brownian motion with Hurst parameter $H$, \ie ,\ the centred Gaussian process with covariance function $K_H : (s,t) \in \R_+^* \times \R_+^* \mapsto \frac{1}{2}(|t|^{2H} + |s|^{2H} - |t-s|^{2H})$. Then the solution to the mimicking problem is given by the centred Gaussian process $X'$ with covariance function \[K_{H}' : (s,t) \in \R_+^* \times \R_+^*  \mapsto 
	\begin{cases}
		s^{H}t^H\mathds{1}_{s = t} + 0\mathds{1}_{s \neq t} & \text{if } H < 1/2  \\ \min(s,t) &\text{if } H=1/2 \\   s^{H}t^{H}& \text{if } H > 1/2 
	\end{cases} .\]
	Indeed, given $t>0$ and $h>0$, the correlation between $X_t$ and $X_{t+h}$ reads
	\[
	\rho(t,t+h)
	=\frac{K_H(t,t+h)}{\sqrt{K_H(t,t)}\,\sqrt{K_H(t+h,t+h)}}
	=\frac{\frac12\bigl(t^{2H}+(t+h)^{2H}-h^{2H}\bigr)}{t^H (t+h)^H},
	\]
	so after a straightforward computation,
	
	\[
	h^{-1}(1-\rho(t,t+h))\sim \frac{1}{2h}\left(\frac{h}{t}\right)^{2H},
	\qquad h\to 0^+.
	\]
	Therefore the instantaneous decorrelation rate is
	\[
	\alpha^{K_H}(t)=
	\begin{cases}
		+\infty, & \text{if } H<\frac12\\[4pt]
		\dfrac{1}{2t}, & \text{if } H=\frac12\\[8pt]
		0, & \text{if } H>\frac12
	\end{cases},
	\]
	and $$\sqrt{K(s,s)K(t,t)} \exp\left(-\int_s^t \a^K(u) du\right) = \begin{cases}
		s^{H}t^H\mathds{1}_{s = t} + 0\mathds{1}_{s \neq t} & \text{if } H < 1/2  \\ \min(s,t) &\text{if } H=1/2 \\   s^{H}t^{H}& \text{if } H > 1/2 
	\end{cases},$$ which concludes by Proposition \ref{proA:mimicking} (see Remark \ref{remq:generalization} for the case $\alpha \equiv +\infty$).
\end{ex}

This shows that the mimicking process of fractional Brownian motion exhibits three radically different behaviours depending on the Hurst parameter. In the regime $H<\tfrac12$, the instantaneous decorrelation rate is infinite, and the mimicking process degenerates into a process with independent values at distinct times. At the critical value $H=\tfrac12$, the mimicking process coincides with standard Brownian motion. In the regime $H>\tfrac12$, the Markov transform corresponds to a perfectly correlated process.

The following result is the first main result of the paper: it states that when the decorrelation rate is regular enough, there exists a unique Markov transform, which coincides with the solution of the mimicking problem of Proposition \ref{proA:mimicking}.

\begin{theoremA}\label{them:transform_is_mimicking}
	Let $X$ be a Gaussian process with continuous covariance $K$ function and positive variance function. Assume $\alpha^K$  satisfies 
	\begin{equation}\label{eq:univ_conv_bis}
		\sup_{v \in  [s,t] } \left| \a^K(v) - L_v^{K}(h) \right| \dcv{h \to 0^+} 0.
	\end{equation} 
	 Then, the mimicking problem of Proposition \ref{proA:mimicking} admits a unique solution $Y$, and for all admissible sequence of sets of times $(R_n)_{n \geq 1}$, the sequences $(X^{R_n})_{n \geq 1}$ converge weakly to $Y$ on compact sets\footnote{By weak convergence on compact sets we mean that for all $a< b \in \R^2$, the law of the process $(X_t^{R_n})_{t\in [a,b]}$ weakly converges to the law of $(Y_t)_{t \in [a,b]}$, where weak convergence on $\mathcal P(\mathcal C([a,b],\mathbb R))$ refers to
 	convergence against bounded continuous functional on
	 $\mathcal C([a,b],\mathbb R)$ endowed with the uniform topology. Processes are understood up to modification; in particular, the statement
	 requires the existence of continuous realizations of the processes involved,
	which will be established in the proof of the result.}.
\end{theoremA}

We now do a formal computation to provide some intuition on the proof of this result and explain why the instantaneous decorrelation rate arises as a natural measure of regularity in the context of Markov transformation of Gaussian processes. The rigorous counterpart of the computation will be given in Lemma \ref{them:Markov_stationnaire}.

\textbf{Formal computation:}
	First, as we will prove later, the (finite dimensional) convergence $X^{[R_n]} \to X'$ is equivalent to the pointwise convergence of the covariance function $K_n$ of $X^{[R_n]}$ to the covariance function $K'$ of $X'$. Moreover, as we shall also prove later, if $X$ is centred with constant variance $1$, for all $s<t \in \R^2$, $$K_n(s,t) = K(t_0^n,t_1^n) \dots K(t_{m_n-1}^n,t_{m_n}^n),$$ where $\{t_0^n<\dots<t_{m_n}^n\} := \{s,t\}\cup \left( R_n \cap ]s,t[ \right) $. If the mesh of $R_n$ is going to $0$, then $h_i^n := t_i^n - t_{i-1}^n$ is small and $K(t_{i-1}^n,t_i^n)$ is near to $K(t_{i}^n,t_i^n) =1$. Using the approximations $$ x \sim_{x \to 1} e^{x-1}  \qquad \text{and } \qquad  L_{t_{i-1}^n}(h_i^n)\approx \alpha(t_{i-1}^n)$$ we get $K(t_{i-1}^n,t_i^n) \approx e^{K(t_{i-1}^n,t_i^n)-1} = e^{- h_i^n  L_{t_{i-1}^n}(h_i^n)} \approx e^{-h_i^n\a (t_{i-1}^n)}$, and using the usual approximation of the integral by Riemann sums, we obtain
	\begin{equation*}
		K_n(s,t) \approx \exp\left( -\sum_{i=1}^n h_i^n \alpha(t_{i-1}^n) \right) \approx \exp\left( -\int_s^t \alpha(u) du\right).
	\end{equation*}   
\vspace{0.2cm}

Note that under the Assumption of Theorem \ref{them:transform_is_mimicking}, for all admissible sequence of times $(R_n)_{n \geq 1}$, the processes $X^{R_n}$, $n \geq 1$ converge to the solution $Y$ of the mimicking problem with respect to the finite dimensional convergence. In particular $X$ admits a unique Markov transform. However, as mentioned above, there is no reason why a process should admit a Markov transform and, if it does, no reason why it should be unique. Regarding the question of existence, we prove that given any (non-singular) Gaussian process there exists a \emph{local} Markov transform.  
\begin{defi}
	We say that $X'$ is a local Markov transform of $X$ if, for each pair of times $s<t \in \R^2$, there exists a sequence of partitions $(R_n^{s,t})_{n \geq 1}$ of $[s,t]$ with mesh size going to $0$ and such that the laws of $\left( X^{R_n^{s,t}}_s,X^{R_n^{s,t}}_t \right)_{n \geq 1}$ converge to the law of $(X_s',X_t').$
\end{defi}
This local version of Markov transforms is less stringent than the former global versions, as we just ask for the convergence of the two-dimensional laws, and more importantly, the dependence on $(s,t)$ of the sets of times is allowed. As the following proposition attests, this local version, although less intuitive, is more tractable in general and allows to prove some existence results. Observe that, in contrast with the setting of Theorem \ref{them:transform_is_mimicking}, and its analogue for increasing processes (see \cite[Theorem 2.26]{boubel_markov-quantile_2022}), our existence result holds in higher dimension.

\begin{proA}\label{proA:existence}
	Consider $d \geq 1$ and let $(X_t)_{t \in \R}$ be a Gaussian process such that $\textnormal{Var}(X_t)$ is a non-singular matrix for all $t \in \R$. Then $(X_t)_{t \in \R}$ admits a local Markov transform.
\end{proA}
Regarding uniqueness, Boubel and Juillet asked whether the local Markov transform of a process (when it exists) is always unique, \cite[$\mathsection 5.5.1$, Open Question $a)$]{boubel_markov-quantile_2022}. The following result answers this question negatively:  there exists a (stationary) Gaussian process with an infinity of Markov transforms.
\begin{theoremA}\label{them:counterexample_intro}
	There exists a stationary Gaussian process $(X_t)_{t \in \R}$ whose set of weak local Markov transforms is the set of all  the processes $(Y_t)_{t \in \R}$ satisfying:
	\begin{longlist}
		\item The process $(Y_t)_{t \in \R}$ is centered Gaussian with constant variance function equal to $1$;
		\item The covariance function of $(Y_t)_{t \in \R}$ is non-negative;
		\item The process $(Y_t)_{t \in \R}$ is Markov.
	\end{longlist}
	Moreover, for all $\a \in [0,+\infty]$, the centred Gaussian process $(X_t^\a)_{t \geq 0}$ with covariance function $K_\a: (s,t) \mapsto \exp(-\alpha|t-s|)$ is a Markov transform of $X$. 
\end{theoremA}
Note that this implies that stationarity is not preserved for local Markov transformation: a local Markov transform of a stationary process is not necessarily itself stationary. This is in contrast with Proposition \ref{pro:Gaussianite_limite}, where we show that Gaussianity is preserved through local Markov transformation: a local Markov transform of a Gaussian process is still Gaussian.

\subsection{Organization of the article}
In Section \ref{sec:preliminaries}, we first gather the notation that will be used throughout the article and give some definitions. Then, we thoroughly define the concepts of weak local Markov transform and strong local Markov transform of a measure, and examine how these notions behave under certain transformations. Finally we introduce the notion of weak global Markov transform and strong global Markov transform and relate them to their local counterparts.

Section \ref{sec:proof_mimicking} is dedicated to the proof of Proposition \ref{proA:mimicking}.

Section \ref{sec:Identification_criteria} is dedicated to the proof of Theorem \ref{them:transform_is_mimicking}. In the first subsection we establish the convergence for finite dimensional distributions, while in the second part we use a tightness criterion to go from the convergence of finite dimensional distribution to weak convergence. We finally provide an additional example.

Section \ref{seq:compo_gaussian} is dedicated to the proof of Proposition \ref{proA:existence}. We first show that the concatenation of Gaussian measures is a continuous operation (Lemma \ref{lem:compo_gaussien_stab}). Relying on this lemma, we then apply a generalization of the Kolmogorov extension theorem by Kellerer \cite[Theorem $1$]{kellerer_markov-komposition_1972} to prove Proposition \ref{proA:existence}.

Section \ref{sec:counterexample} is devoted to the proof of Theorem \ref{them:counterexample_intro}. First, we prove a sufficient criterion to identify weak local Markov transforms of a \emph{stationary} Gaussian process: If $\a$ is a cluster point of the decay rate of the correlation function at $0^+$, then the (renormalized) stationary Ornstein-Uhlenbeck process with parameter $\a$ is a weak local Markov transform of $X.$ Then, we use the Weierstrass continuous nowhere-differentiable functions \cite{hardy_weierstrasss_1916} and a theorem of Bochner characterizing covariance function of stationary Gaussian processes \cite{bochner_monotone_1933} to construct a process satisfying the conditions of Theorem \ref{them:counterexample_intro}.

\section{Notation, definitions and basic fact on Markov transformation}\label{sec:preliminaries}

\subsection{Notation}

We gather here some notation that will be used throughout the article.

\begin{nott}[General notation]\label{not:basique}
	\begin{enumerate}
		\item In this article $\N^*$, $\R^*_+$ and $\R^*$ denote the set of positive integers, positive real numbers and non-null real numbers, respectively. We write $\{t_1 < \dots < t_m\}$ and $(t_1< \dots < t_m)$ for the set $\{t_1, \dots , t_m\}$ and the sequence $(t_1,\dots,t_m)$, respectively, when $t_1 < \dots < t_m.$
		\item Given a measurable space $(E,\SS)$, let $\PP(E)$, $\MM_+(E)$ denote the set of probability measures and positive finite measures on $E$, respectively. We endow every product space of measurable spaces with its cylindrical $\sigma$-algebra and every topological space with its Borel sets. If $f : E \to F$ is a measurable function and $\g \in \PP(E)$, let $f_{\#} \g : A \mapsto \g\left(f^{-1}(A)\right)$ denote the push-forward measure of $\g$ by $f$.
		\item Given a family of Polish spaces $(E_t)_{t \in T}$ and $T' \subset T$, we denote by $\proj^{T'}$ the projection from $\prod_{ t \in T} E_t $ to $\prod_{ t \in T'} E_t$. Then, for all $P \in \PP(\prod_{t \in T} E_{t})$, define  $$P^{T'} = \proj^{T'}_\# P \in \PP(\prod_{t \in T'} E_{t}).$$ In case $T' = \{t_1,\dots,t_m\}$, we rather denote $P^{T'}$ by $P^{t_1, \dots , t_m}$. Furthermore, for any family $(\mu_t)_{ t \in T} \in \prod_{t \in T} \PP(E_t)$, we define $$\Marg((\mu_t)_{t \in T}) = \ens{P \in \PP\left(\prod_{t \in T} E_{t}\right)}{\forall t \in T, P^t = \mu_t},$$ and denote by $\bigotimes_{t \in T} \mu_t$  the product of the measures $\mu_t$, $t\in T$.
	\end{enumerate}
	  
\end{nott}
\begin{notdef}[Concatenation and composition] \label{def:compo}
	\begin{enumerate}
		\item Let $E_1, \dots, E_d$ be Polish spaces, $\mu_i \in \PP(E_i)$ for $i = 1,\dots,d$, and $P_i \in \Marg(\mu_i,\mu_{i+1})$, for $i=1, \dots,d-1$. The concatenation of $P_1, \dots, P_{d-1}$ is the probability measure $P_1 \circ \dots \circ P_{d-1} \in \PP(E_1 \times \cdots \times E_d)$ defined by \[ (P_1 \circ \dots \circ P_{d-1})(A_1 \times \dots \times A_d) = \int_{A_1 \times \dots \times A_d} \mu_1(dx_1)k^{1,2}(x_1,dx_2)\dots k^{d-1,d}(x_{d-1},dx_d), \]
		where $k^{i,i+1}$ is the probability kernel defined by the disintegration $P_i(dx_i,dx_{i+1}) = \mu_i(dx_i)k^{i,i+1}(x_i,dx_{i+1})$. If $P_{1}(dx_1,dx_2) = \mu_2(dx_2)k^{2,1}(x_2,dx_1)$, the reader may verify that $$(P_1\circ P_{2})(dx_1,dx_2,dx_3) = \mu_2 (dx_2) [k^{2,1}(x_2, \cdot) \otimes k^{2,3}(x_2, \cdot)](dx_1,dx_3).$$ This means that, conditionally on the present, the future is independent of the past.
		\item For $T \subset \R$ and $(E_t)_{t \in T}$ a family of Polish spaces, we say that a probability $P \in \PP( \prod_{t \in T} E_{t})$  is a Markov measure if for every subset  $\{t_1< \dots <t_m\} \subset T$, $P^{t_1,\dots,t_m} = P^{t_1,t_2} \circ \cdots \circ P^{t_{m-1},t_m}.$ Of course the notion of Markov  measure is related to the more usual notion of Markov process: the reader may verify that a process is Markov (with respect to its canonical filtration) if, and only if, its law is a Markov measure.
		\item The composition $P_1  \cdot \ \dots \ \cdot  P_d \in \PP(E_1 \times  E_d)$ of $P_1, \dots, P_d$ is now defined by $$P_1 \cdot \ \dots \ \cdot P_d  = \proj^{1,d}_\# (P_1 \circ \dots \circ P_d).$$
		\item For $s<t \in \R^2$, we write $\SS_{[s,t]}$ the set of partitions of $[s,t]$, \ie , \ the sequences $(t_1 < \dots < t_m) \in \R^m$ with $t_1 = s$ and $t_m = t.$ If $R = (t_1 < \dots < t_m) \in \SS_{[s,t]}$, we set $P_{\{R\}}^{s,t} := P^{t_1,t_2} \cdot \ \dots \ \cdot P^{t_{m-1},t_m}$ and denote by $\sigma_R := \sup_{i = 1,\dots,d-1} |t_{i+1} - t_i|$ the mesh of $R$. 
	\end{enumerate}

\end{notdef}

For $T \subset \R $ and $K : T \times T \to \R$, we recall  that $K$ is said to be a positive semi-definite kernel if for all $(t_1,\dots,t_m) \in T^m$, the matrix $(K(t_i,t_j))_{1 \leq i,j \leq m}$ is symmetric and positive semi-definite. Moreover, the kernel $K$ is said to be stationary if there exists $\ti{K} :\R \to \R$ such that, for all $(s,t) \in T \times T$, $ K(s,t) = \ti{K}(t-s)$.  It is a standard fact that a function  $K : T \times T \to \R$ is a positive semi-definite  kernel (resp. stationary positive semi-definite  kernel) if and only if there exists a Gaussian measure (resp. a stationary Gaussian measure) on $\R^T$  with covariance function $K$. For a proof, one can e.g. refer to \cite[Chapter $3$]{von_mises_mathematical_1964}.

\begin{nott}[Positive semi-definite kernel]
	Let  $K : T \times T \to \R$ denote a positive semi-definite  kernel.
	\begin{enumerate}
		\item We denote by $v_K : t \in T \mapsto K(t,t) \in \R$ its variance function and $\s_K := \sqrt{v_K}$ its standard deviation function. The kernel $K$ is then said non-singular if its variance function takes positive values.
		\item If $K$ is non-singular, we denote by $c_K :(s,t) \mapsto \s_K^{-1}(s)\s_K^{-1}(t)K(s,t)$ its correlation function.
		\item For all $t \in T \setminus \{\sup T \}$, let $L^K_t :(T-t) \cap \R_+^* \to \R_+$  the decay rate of the correlation of $K$ at point $t$, that is,   $$L^K_t(h) = h^{-1}(1-{c_K}(t,t+h)), \qquad \text{for all }  h \in \R_+^*.$$
		\item If for $t \in T \setminus\{\sup T\}$, $L_t^K(h)$ converges in $[0,+\infty]$ when $h \to 0^+$, we  denote by $\alpha^K(t)$ the instantaneous decorrelation rate of $K$ at time $t$, that is $$\a^K (t) = \lim_{h \to 0^+} L_t^K(h).$$ The map $\a^K :T \setminus \{\sup T\} \to [0,+\infty]$ is then called the instantaneous decorrelation rate of $K$. 
	\end{enumerate}
\end{nott}

If $X = (X_t)_{t \in T}$ is a centered real-valued random process with covariance function $K$, then $v_K$ is the variance function of $X$, $c_K$ is the correlation function of $X$ and $K$ is non-singular if and only if $X_t = 0$ a.s. never happens. Let $\C(X_s,X_t)$ denote the correlation between $X_s$ and $X_t$. Then, the map $L_t^K : h \mapsto L_t^K(h) = h^{-1} \left( \C(X_t,X_t) - \C(X_{t+h},X_t)\right)$ is the decay rate of the function $s \mapsto \C(X_s, X_t)$, that is the decay rate of the correlation to $X_t.$ Hence, $\a^K(t)$ is the instantaneous decay of the correlation to $X_t$, \ie , the instantaneous decorrelation rate of $X$ at time $t.$

\begin{nott}[Gaussian measure]
	\begin{enumerate}
		\item For every $d \geq 1,$ we denote by $\GG_d$  the set of centred Gaussian measures and by $\GG_d^*$ the set of centred Gaussian measures on $\R^d$ with a non-singular covariance matrix.
		\item Consider $(d_1,d_2) \in \N^* \times \N^*$, $\pi \in \PP(\R^{d_1} \times \R^{d_2})$ and $(X_1,X_2)$ a $(\R^{d_1} \times \R^{d_2})$-valued random variable with law $\pi$. Let $\Sigma_{\pi}^{d_1,d_2} \in \MM_{d_1,d_2}(\R)$ denote the covariance between $X_1$ and $X_2.$
		\item Given $\g_1 \in \PP(\R^n)$ and $\g_2 \in \PP(\R^m)$, we denote by $\GG(\g_1,\g_2)$ the set of Gaussian transport plan from $\g_1$ to $\g_2$, that is, $\GG(\g_1,\g_2) = \Marg(\g_1,\g_2) \cap \GG_{m+n}$.
		\item For  $\g \in \PP(\R^d)$, let $m_\g \in \R^d$ (resp. $\Sigma_{\g} \in \MM_d(\R)$) denote the mean value of $\g$ (resp. the covariance of $\g$), \ie , \ the mean value (resp. the covariance matrix) of a $\R^d$-valued random variable with law $\g.$
	\end{enumerate}
\end{nott}

For the sake of completeness, we collect here some notation referring to objects that will be defined later in the paper.
\begin{nott}
	Given a measurable map $\a : \R \to [0,+\infty]$, we denote by $K_\alpha$ the map defined by $K_\alpha(s,t) = \exp\left(- \int_{\min(s,t)}^{\max(s,t)} \a(u) du \right)$ and by $P_\alpha$ the only centred Gaussian measure with covariance function $K_\a$ (see  Proposition \ref{pro:alpha_proces}).
\end{nott}

\begin{nott}
	Given a measure $P \in \PP(\R^\R)$ and a finite set of times $R \subset \R$, we denote by $P_{[R]}$ the measure $P$ made Markov at times $R$ (see Definition \ref{def:made_markov}). A process with law $P_{[R]}$ will usually be denoted $X^R$. The class of admissible set of times (see Definition \ref{def:admissible}) is denoted by $\mathcal{A}$.
\end{nott}

\subsection{Markov transformation}

We now rigorously define different notions of Markov transform that will arise in the paper. We start by defining the notion of (weak and strong) local Markov transform of a measure. We emphasize that it is different from the notion of \emph{global} Markov transform introduced in Definition \ref{defi:global_markov_transform}.

\begin{defi}[Local Markov transform]\label{def:Markovinified}
	Consider an interval $T \subset \R$, $d \geq 1$  and a measure $P \in \PP( (\R^d )^T )$.
	\begin{longlist}
		\item We say that $P$ admits a weak local Markov transform if there exists a Markov measure $P'$ such that \[\forall s<t \in T^2, \exists (R_n)_{n \geq 1} \in {\left(\SS_{[s,t]} \right)}^{\N^*}, \accol{\lim_{n \to +\infty} P_{\{R_n\}}^{s,t} = P'^{s,t} \\  \lim_{n \to + \infty} \s_{R_n} = 0}.\] In this case, we say that $P'$ is a weak local Markov transform of $P.$
		\item We say that $P$ admits a strong local Markov transform if there exists a Markov measure $P'$ such that   \[ \forall s<t \in T^2, \forall (R_n)_{n \geq 1} \in {\left(\SS_{[s,t]} \right)}^{\N^*} :  \lim_{n \to +\infty} \s_{R_n} = 0 \implies \lim_{n \to + \infty} P_{\{R_n\}}^{s,t} = P'^{s,t}.\] In this case, we say that $P'$ is a strong local Markov transform of $P.$
		\item Given two $\R^d$-valued stochastic processes $(X_t)_{t \in T}$ and $(Y_t)_{t \in T}$, we say that $(Y_t)_{t \in T}$ is a weak (resp. strong) local Markov transform of $(X_t)_{t \in T}$ if the law of $(Y_t)_{t \in T}$ is a weak (resp. strong) local Markov transform of the law of $(X_t)_{t \in T}.$
	\end{longlist}	
\end{defi}
\begin{remq}\label{remp:uniqueness_markov_transform}
	\begin{longlist}
		\item If $P'$ is a strong local Markov transform of a measure $P \in \PP( (\R^d )^T )$ and $Q$ is a weak local Markov transform of $P$, then $P' = Q.$ Indeed, for  every $s<t \in T^2$, there exists $(R_n)_{n \geq 1} \in \left( \SS_{[s,t]} \right)^{\N^*}$ such that $Q^{s,t} = \lim_{n \to + \infty} P^{s,t}_{\{R_n\}}$ and $\lim_{n \to + \infty} \s_{R_n} = 0.$ Since $P'$ is a strong local Markov transform, it follows that $ {P'}^{s,t} =\lim_{n \to + \infty} P^{s,t}_{\{R_n\}}$, which implies ${P'}^{s,t} = Q^{s,t}$. Since a Markov measure is completely characterized by its two-dimensional laws, we obtain $P' = Q.$ In particular, there is at most one strong local Markov transform. Thus if a measure $P$ admits a strong local Markov transform, we refer to \emph{the} strong local Markov transform of $P.$
		\item A strong local Markov transform of $P$ is clearly a weak local Markov transform of $P$, but the converse is false, even when the weak local Markov transform is unique. For instance, if $(Y_t)_{t \in \R}$ is a family of i.i.d. random variables with law $\NN(0,1)$ and $P$ is the law of the process $ X = (Y_0 1_{t \in \Q} + Y_t 1_{t \in \R \setminus \Q} )_{t \in \R}$, the reader may verify that $P' := \otimes_{t \in \R} \ \NN(0,1)$ is the unique weak local Markov transform of $P$, but is not a strong local Markov transform of $P$.
		\item If $P$ is a Markov process, then $P_{\{R\}}^{s,t} = P'^{s,t}$ for all $R \in \SS_{[s,t]}$. This implies $P$ is a strong local Markov transform of $P$ and the only weak local Markov transform of $P.$ 
	\end{longlist}
	
\end{remq}
 The following proposition indicates how Markov transforms behave with respect to a change in time and a componentwise transformation. Given $S,T \subset \R$, $d \geq 1$, $P \in \PP( ( \R^d )^T )$ and $\phi :S \to T$, let $P^{\phi}$ denote the measure on $\R^S$ such that, for all $s_1 < \dots < s_m \in S^m$, $(P^{\phi})^{s_1,\dots,s_m} := P^{\phi(s_1), \dots, \phi(s_n)}$. If $(X_t)_{t \in T}$ is a stochastic process with law $P$, then $P^{\phi}$ is the law of the time-changed process $\left(X_{\phi(s)}\right)_{s \in S}.$
\begin{pro}\label{pro:transformation}
	Consider $T \subset \R$ an interval, $d \geq 1$ and $P, P' \in \PP((\R^d)^T)$.
	\begin{longlist}
		\item Let $(f_t)_{t \in T}$ be a family of continuous injective functions from $\R^d$ to $\R^d$ and set $  f := \left( \otimes_{t \in T} f_t \right): (x_t)_{t \in T} \in {(\R^d)}^T \mapsto (f_t(x_t))_{t \in T} \in {(\R^d)}^T$.  If $P'$ is a weak (resp. the strong) local Markov transform of $P$, then ${f }_\# P'$ is a weak (resp. the strong) local Markov transform of ${ f}_\# P$.
		\item Let $S \subset \R$ be an interval and $\phi :S \to T$ be a strictly increasing continuous function. If $P'$ is a weak (resp. the strong) local Markov transform of $P$, then $P'^{\phi} \in \PP((\R^d)^S)$ is a weak (resp. the strong) local Markov transform of $P^\phi \in \PP((\R^d)^S)$.
	\end{longlist} 
\end{pro}
\begin{proof}
	\begin{longlist}
		\item Fix $s<t \in \R^2$. The reader may verify that, due to the injectivity of the functions $(f_t)_{t \in \R}$, the Markov property is preserved from $P'$ to $f_\# P'$ and that, for all $ Q \in \PP( (\R^d)^T )$, $R \in \SS_{[s,t]}$, we have  $\left(f_\# Q \right)_{\{R\}}^{s,t} = (f_s \otimes f_t)_\# Q_{\{R\}}^{s,t}$ . Therefore, if $\lim_{n \to + \infty} P_{\{R_n\}}^{s,t} = P'^{s,t}$, the continuity of $f_s \otimes f_t$ implies $\left(f_\# P \right)_{\{R_n\}}^{s,t} = (f_s \otimes f_t)_\# P_{\{R_n\}}^{s,t} \dcv{n \to +\infty} (f_s \otimes f_t)_\# P'^{s,t} = \left(f_\# P' \right)^{s,t}$. This shows the result.
		\item First, assume $P'$ is the \emph{strong local Markov transform} of $P$, fix $s<t \in S^2$ and $(R_n)_{n \geq 1} \in \left(\SS_{[s,t]}\right)^{\N^*}$ such that $\lim_{n \to +\infty}\s_{R_n} = 0.$ Since $\phi$ is strictly increasing and uniformly continuous on $[s,t]$, $(\phi(R_n))_{n \geq 1} \in \left( \SS_{[\phi(s),\phi(t)]} \right)^{\N*}$ and satisfies $ \lim_{n \to +\infty} \s_{\phi(R_n)} = 0.$ Thus  $(P^{\phi})_{\{R_n\}}^{s,t} = P^{\phi(s),\phi(t)}_{\{\phi(R_n)\}} \dcv{n \to + \infty} \left(P'\right)^{\phi(s),\phi(t)} = {\left(P^{\phi}\right)}^{s,t}$, which shows that $P^\phi$ is the strong local Markov transform of $P.$ Next, assume $P'$ is a \emph{weak local Markov transform} of $P$ and fix $s<t \in S^2$. We must find a sequence $(R_n)_{n \geq 1} \in \left( \SS_{[s,t]}\right)^{\N^*}$ such that $\lim_{n \to + \infty} \s_{R_n} = 0$ and $\lim_{n \to + \infty} \left(P^{\phi}\right)_{\{R_n\}}^{s,t} = \left(P'^{\phi}\right)^{s,t}.$ As $\phi(s) < \phi(t)$  and $P'$ is a weak local Markov transform of $P$, there exists  a sequence $(S_n)_{n \geq 1} \in \left(\SS_{[\phi(s),\phi(t)]}\right)^{\N^*}$ such that $\lim_{n \to +\infty}\s_{S_n} = 0$ and $\lim_{n \to +\infty} P^{\phi(s),\phi(t)}_{\{S_n\}} = P'^{\phi(s),\phi(t)} = {\left(P'^{\phi}\right)}^{s,t}.$ We set $R_n := \phi^{-1}(S_n)$. Since $\phi^{-1} : [\phi(s),\phi(t)] \to [s,t]$ is strictly increasing, uniformly continuous and $\left(P^{\phi}\right)_{\{R_n\}}^{s,t} = P_{\{S_n\}}^{\phi(s),\phi(t)}$,  the sequence $(R_n)_{n \geq 1}$ satisfies the conditions. \qedhere
	\end{longlist}
\end{proof}

\begin{remq}\label{remq:iden_mark_trans_def}
	\begin{itemize}
		\item In terms of processes, Proposition \ref{pro:transformation} implies that if $(Y_t)_{t \in T}$ is a weak (resp. the strong) local Markov transform of  $(X_t)_{t \in T}$, then $\left(f_{\phi(s)}(Y_{\phi(s)})\right)_{s \in S}$ is a weak (resp. the strong) local Markov transform of $\left(f_{\phi(s)}(X_{\phi(s)})\right)_{s \in S}$.
		\item Consider $S,T$ two intervals, $u :S \to \R^*$ and $\phi : S \to T$ a strictly injective continuous function. Assume $P \in \PP(\R^T)$ is a centered Gaussian process with covariance function $K$ and $P' \in \PP(\R^S)$ is a weak (resp. the strong) Markov transform of $P$ with covariance function $K'$. Applying Proposition \ref{pro:transformation}, it is straightforward that the centered Gaussian process with covariance $ (s,t) \in S^2 \mapsto u(s)u(t)K'(\phi(s),\phi(t))$ is a weak (resp. the strong) local Markov transform of  the centered Gaussian process with covariance function $(s,t) \in S^2 \mapsto u(s)u(t)K(\phi(s),\phi(t)).$
	\end{itemize} 
\end{remq}

\begin{remq}\label{remq:centered_enough}
	In the remainder of the article, we assume that our process $P \in \Marg((\mu_t)_{t \in T})$ has non-singular marginals, \ie , \ for all $t \in T,$ $\Sigma_{\mu_t} \in \GL_d(\R)$. In this case, it is possible to apply Proposition \ref{pro:transformation} with $f_t : x_t \mapsto \Sigma_{\mu_t}^{-1/2}(x_t - m_{\mu_t})$ and $f_t^{-1} : u_t \mapsto \Sigma_{\mu_t}^{1/2}u_t + m_{\mu_t}$, where $\Sigma_{\mu_t}^{1/2}$ stands for the symmetric and positive-definite square root of $\Sigma_{\mu_t}$. Thus  $P' \in \Marg((\mu_t)_{t \in T})$ is a weak (resp. the strong) Markov transform of $P$ if and only if ${\left(\otimes_{t \in T} f_t\right)}_\# P'$ is a weak (resp. the strong) Markov transform of  ${\left(\otimes_{t \in T} f_t\right)}_\# P$. Since $ {\left(\otimes_{t \in T} f_t\right)}_\# P' \in \Marg((\ti{\mu}_t)_{t \in T})$, where $\ti{\mu}_t := {f_t}_\# \mu_t$ satisfies $m_{\ti{\mu}_t} = 0$ and $\Sigma_{\ti{\mu}_t} = I_d$, this allows us to restrict our analysis to Markov transforms of centered processes with constant covariance function equal to the identity matrix. 
\end{remq}

We now define rigorously the notion of (weak and strong) global Markov transform already informally defined in the introduction.
We begin by defining the notion of measure made Markov at times $R$, that has been introduced by Boubel and Juillet  \cite[Definition $4.18$]{boubel_markov-quantile_2022}. First we  extend the notation $\circ$ to measures on product space that have a common extremal marginal.

\begin{nott}[Marginal gluing]\label{nott:marginal_gluing}
	\begin{enumerate}
		\item Consider $T \subset \R$ and Polish spaces $F_s$, $s \in T$. Let $S_1 \subset T$, $S_2 \subset T$ such that $\sup (S_1) = \inf (S_2) =: a$, and $P_1 \in \PP(\prod_{s \in S_1} F_s)$, $P_2 \in \PP(\prod_{s \in S_2} F_s)$ such that $P^a_1 = P_2^a$. Then $P_1 \circ P_2$ denotes the gluing at time $a$: formally, this correspond to the usual concatenation of $P$ and $Q$ in Definition \ref{def:compo} with $d = 3$, $E_1 = \prod_{s \in S_1 \setminus \{a\}} F_s$, $E_2 = F_a$ and $E_3 = \prod_{s \in S_2 \setminus \{a\}} F_s$.
		\item More generally, given a family $(S_i)_{i= 1,\dots,d}$ of subsets of $T$ such that $\sup(S_{i-1}) = \inf(S_i)=:a_i$ and probabilities $P_i \in \PP(\prod_{s \in S_i} E_s)$ such that $P_{i-1}^{a_i} = P_i^{a_i}$, the sequence $P_1 \circ \dots \circ P_d$ is defined recursively using the previous point:
		$$ P_1 \circ \dots \circ P_i := (P_1 \circ \dots \circ P_{i-1}) \circ P_i.$$ 
	\end{enumerate}
\end{nott}

\begin{defi}[Measure made Markov at times $R$]\label{def:made_markov}
	Let $T \subset \R$ be a set, $(E_t)_{t \in T}$ a family of Polish spaces, $P \in \PP(\prod_{t \in T} E_t)$ a probability measure, and $R = \{r_1 < \dots < r_m\} \subset T$ a finite set of times. The measure $P$ made Markov at times $R$, denoted $P_{[R]}$ is defined by 
	\begin{equation}
		P_{[R]} = P^{]-\infty,r_1] \cap T} \circ P^{[r_1,r_2] \cap T} \circ \  \dots \  \circ P^{[r_{m-1},r_m] \cap T} \circ P^{[r_m,+\infty[ \cap T},
	\end{equation}
	where the concatenation is defined as in Notation \ref{nott:marginal_gluing}. Given a process $X$ with law $P$, we say that $X^R$ is the\footnote{We refer to \emph{the} transformation, even if we just have uniqueness in law.} process $X$ made Markov at times $R$ if $X^R$ has law $P_{[R]}.$
\end{defi}

\begin{remq}
	\begin{enumerate}
		\item In \cite[Definition $4.18$]{boubel_markov-quantile_2022}, the authors start by constructing a consistent family of finite-dimensional marginals, and then define the measure made Markov at times $R$ using the Kolmogorov extension theorem. More precisely, for each finite set $S \subset T$, writing \[S \cup R = \{s_1^0 < \dots < s_{k_0}^0 < r_1 < s_1^1 < \cdots < s_{k_1}^1 < r_2< \cdots < r_m < s_1^m < \cdots < s_{k_m}^m\}\] they define \[\mu_S := \proj^S_{\#}[ P^{s_1^0,\dots,s^0_{k_0},r_1} \circ P^{r_1, s_1^1 , \dots, s^1_{k_1},r_2} \circ \dots \circ P^{r_m,s_1^m,\dots, s_{k_m}^m}],\] where the  concatenation is defined as in Notation \ref{nott:marginal_gluing}. Letting $\SS$ denote the class of finite subsets of $T$, $(\mu_S)_{S \in \SS}$ is a consistent family of measures. The measure $P$ made Markov at times $R$ is then the unique measure $Q$ such that $\proj^S_\# Q = \mu_S$ for all $S \in \SS$, which exists by the Kolmogorov extension theorem. This coincides with our definition as the measure $P_{[R]}$ defined in Definition \ref{def:made_markov} satisfies $\proj^S_\# P_{[R]} = \mu_S$ for all $S \in \SS$.
		\item To provide additional intuition for the process made Markov at times $R$, we describe explicit sampling procedures for it. Let $(X_t)_{t \in \R}$ be a process and fix $R = \{r_1,\dots,r_m\} \subset \R$. We explain how to sample a realization $(y_t)_{t \in \R}$ of the process $X^R$.\\
		\textbf{First sampling procedure.}
		\begin{itemize}
			\item Sample $(y_t)_{t \le r_1} \sim \Law((X_t)_{t \le r_1})$;
			\item For $i=2,\dots,m$, sample $(y_t)_{r_{i-1}<t \le r_i} \sim \Law((X_t)_{r_{i-1}<t \le r_i}\mid X_{r_{i-1}} = y_{r_{i-1}})$;
			\item Sample $(y_t)_{t>r_m} \sim \Law((X_t)_{t > r_m}\mid X_{r_m} = y_{r_m})$.
		\end{itemize}
		\textbf{Second sampling procedure.}
		\begin{itemize}
			\item Sample $(y_{r_1},y_{r_2}) \sim \Law(X_{r_1},X_{r_2})$, and for $i=3,\dots,m$ sample 
			$y_{r_i} \sim \Law(X_{r_i}\mid X_{r_{i-1}} = y_{r_{i-1}})$;
			\item Sample $(y_t)_{t<r_1} \sim \Law((X_t)_{t<r_1}\mid X_{r_1} = y_{r_1})$, 
			$(y_t)_{r_{i-1}<t<r_i} \sim \Law((X_t)_{r_{i-1}<t<r_i}\mid X_{r_{i-1}}=y_{r_{i-1}}, X_{r_i}=y_{r_i})$, 
			and $(y_t)_{t>r_m} \sim \Law((X_t)_{t>r_m}\mid X_{r_m}=y_{r_m})$.
		\end{itemize}
	\end{enumerate}
\end{remq}

To simplify the following definitions and notation, we assume that $T = \R$, although the results hold for any interval $T \subset \R.$ A global Markov transform will be defined as a Markov limit of a sequence $P_{[R_n]}$ for an admissible sequence $(R_n)_{n \geq 1}$ of set of times (see Notation \ref{def:admissible}). We first consider the topology of finite-dimensional convergence, namely, weak convergence on $\PP(\prod_{t \in \R}E_t)$, where $\prod_{t \in \R}E_t$ is endowed with the product topology. More explicitly, a sequence of elements $P_n \in \PP(\prod_{t \in \R}E_t)$, $n \ge 1$ converges to $P \in \PP(\prod_{t \in \R}E_t)$ with respect to this topology if for all $s_1< \cdots < s_m \in \R^m$, $(P^{s_1,\dots,s_m}_n)_{n \geq 1}$ weakly converges to $P^{s_1,\dots,s_m}$ with respect to the weak topology. This convergence is denoted by $P_n \xrightarrow[n \to + \infty]{\fd} P$. The following remark states that for real-valued Gaussian measures, this convergence is equivalent to two-dimensional convergence.
\begin{remq}\label{remq:finite_dim_conv_gaussian}
	Let $P_n \in  \PP\left( \R^\R \right), n \geq 0$ be centred Gaussian processes. Then $P_n \xrightarrow[n \to + \infty]{\fd} P_0$ if, and only if, for all $s<t \in \R^2$, $\lim_{n \to +\infty} \left(P_n\right)^{s,t} = P_0^{s,t}$. The direct implication is straightforward. Conversely, assume that for all $s<t \in \R^2$, $\lim_{n \to + \infty} \left(P_n\right)^{s,t} = P_0^{s,t}$, and fix  $t_1 < \cdots < t_m \in \R^m$. According to Lemma \ref{lem:stab_gaussienne}, $\lim_{n \to + \infty} P_n^{t_1, \dots, t_m} = P_0^{t_1, \dots,t_m}$ if, and only if, $\Sigma_{P_n^{t_1, \dots, t_m}} \dcv{n \to + \infty}  \Sigma_{P_0^{t_1, \dots,t_m}}.$ Since for all Gaussian process $Q \in \PP(\R^{\R})$ and $i<j \in \{1, \dots, m\}$, we have $\Sigma_{Q^{t_1,\dots,t_m}}(i,j) = \Sigma_{Q^{t_i,t_j}}(1,2)$, we get the converse implication. 
\end{remq} 

\begin{nott}\label{def:admissible}
	Set $\R^k_{\shortuparrow} := \ens{(t_1,\dots,t_k) \in \R^k}{t_1< \dots <t_k}$ and $\R_{\shortuparrow} := \cup_{k \geq 1} \R^k_{\shortuparrow}$. A sequence $(R_n)_{n \geq 1} \in \left(\R_{\shortuparrow}\right)^{\N^*}$ is said admissible if: $$\lim_{n \to + \infty} \sup(R_n) = +\infty, \quad  \lim_{n \to + \infty} \inf(R_n) = -\infty \quad \text{and } \lim_{n \to + \infty }\s_{R_n} = 0.$$ We denote by $\A$ the set of all admissible sequences. 
\end{nott} 

We can now define weak and strong global Markov transform.

\begin{defi}[Global Markov transform]\label{defi:global_markov_transform}
	Let $(E_t)_{t \in \R}$ be a family of Polish space and $P$ a probability measure on $\prod_{t \in \R}E_t.$ 
	\begin{longlist}
		\item We say that $P$ admits a weak global Markov transform if there exists a Markov measure $P'$ on $\prod_{t \in \R}E_t$ and $(R_n)_{n \geq 1} \in \A$ such that $ P_{[R_n]} \xrightarrow[n \to + \infty]{\fd} P'.$ We say that $P'$ is a weak global Markov transform of $P.$
		\item We say that $P$ admits a strong global Markov transform if there exists a Markov measure $P'$ on $\prod_{t \in \R}E_t$ such that for all  $(R_n)_{n \geq 1} \in \A $, we have $P_{[R_n]} \xrightarrow[n \to + \infty]{\fd} P'.$ We say that $P'$ is the strong global Markov transform of $P.$
	\end{longlist}	
\end{defi}

The two-dimensional laws at time $(s,t)$ of the measure made Markov at times $R$  can be obtained by composing the transition kernel passing through the times of $R$. More explicitly, one can readily verify that $\left(P_{[R]}\right)^{s,t} = P_{\{\{s,t\} \cup (R \cap ]s,t[)\}}^{s,t}$ (see Definition \ref{def:compo} for the right-hand side of the equality). In the following proposition, we verify that this implies that a weak (resp. strong) global Markov transform of $P$ is a weak (resp. strong) local Markov transform of $P$. A natural question to ask is whether local Markov transforms are also global Markov transforms. This is less obvious and generally false. However, applying Remark \ref{remq:finite_dim_conv_gaussian}, we shall prove that it holds for strong Markov transforms of Gaussian measures: If a Gaussian measure $P'$ is the strong local Markov transform of $P$, then $P'$ is also the strong global Markov transform of $P$\footnote{According to Remark \ref{remp:uniqueness_markov_transform}, this justifies referring to \emph{the} strong global Markov transform of $P$.}. 

\begin{pro}\label{pro:cov_proc_Markovinifie}
	Let   $P \in \PP(\R^\R)$ denote a Gaussian measure  and $P' \in \PP(\R^\R)$ a Markov measure.
	\begin{longlist}
		\item The measure $P'$ is the strong global Markov transform of $P$ if and only if it is the strong local Markov transform of $P.$
		\item If $P'$ is a weak global Markov transform of $P$, then $P'$ is a weak local Markov transform of $P$.
	\end{longlist}
\end{pro}
\begin{proof}
	\begin{longlist}
		\item Assume $P'$ is the strong \emph{global} Markov transform of $P.$ To show that $P'$ is the strong \emph{local} Markov transform of $P$,  fix $s<t \in \R^2$ and  time set $\ti{R}_n \in \SS_{[s,t]}$ , $n \geq 1$ such that $\lim_{n \to +\infty} \s_{\ti{R}_n} = 0.$ For all $n \geq 1$, define $$R_n = \ti{R}_n \cup \big( ([-n,s[ \cup ]t,n])\cap \left(2^{-n} \mathbb{Z} \right) \big).$$ Since $(R_n)_{n \geq 1} \in \A$, $(R_n \cap ]s,t[ ) \cup \{s,t\} = \ti{R}_n$ and  $P_{[R_n]} \xrightarrow[n \to + \infty]{\fd} P$, we have $ P^{s,t}_{\{\ti{R}_n\}} =  P^{s,t}_{(R_n \cap ]s,t[ ) \cup \{s,t\}} = \left(P_{[R_n]}\right)^{s,t}  \dcv{n \to + \infty} {P'}^{s,t}$. 
		Conversely, assume $P'$ is the strong \emph{local} Markov transform of $P$. By Remark \ref{remq:finite_dim_conv_gaussian} it is sufficient to show that, for all $s<t \in \R^2$ and $(R_n)_{n \geq 1} \in \A$, the convergence $\lim_{n \to + \infty} \left( P_{[R_n]} \right)^{s,t} = (P')^{s,t}$ holds. Defining $\ti{R}_n := \{s,t\} \cup (R_n \cap ]s,t[),$ we have $\ti{R}_n \in \SS_{[s,t]}$ for all $n \geq 1$, and $\lim_{n \to + \infty} \s_{\ti{R}_n} = 0$. Since $P'$ is a strong local Markov transform of $P$,  $\left( P_{[R_n]} \right)^{s,t} =  P_{\{\ti{R}_n\}}^{s,t} \dcv{n \to + \infty} \left(P'\right)^{s,t}$. Thus $P'$ is the strong global Markov transform of $P$ and proves the point.
		\item Let  $(R_n)_{n \geq 1} \in \A $ be such that $P_{[R_n]} \xrightarrow[{n \to + \infty}]{\fd} P'$, fix $s<t \in \R^2$, and define $\ti{R}_n := \{s,t\} \cup (R_n \cap ]s,t[).$ We have $(\ti{R}_n)_{n \geq 1} \in \left(\SS_{[s,t]}\right)^{\N^*}$, $\lim_{n \to + \infty} \s_{\ti{R}_n} = 0$, and $P^{s,t}_{\{\ti{R}_n\}} =  P^{s,t}_{(R_n \cap ]s,t[ ) \cup \{s,t\}} = \left(P_{[R_n]}\right)^{s,t} \dcv{n \to +\infty} \left(P'\right)^{s,t}$. As this holds for all $s<t \in \R^2$, $P'$ is a weak local Markov transform of $P.$ \qedhere
	\end{longlist}
\end{proof}

Note that we did not use the fact that $P$ is Gaussian to prove that every weak (resp. strong) global Markov transform of $P$ is a weak (resp. strong) local Markov transform of $P$. However as we applied Remark \ref{remq:finite_dim_conv_gaussian}, which is only valid in the context of Gaussian measures , we used it to prove that strong local Markov transforms are also strong global Markov transforms.

\section{Proof of Proposition \ref{proA:mimicking}}\label{sec:proof_mimicking}

To prove Proposition \ref{proA:mimicking} we start by verifying that, given a non-negative measurable map $\alpha$, there exists a Markov Gaussian process $P_\alpha$ whose covariance function $K_\alpha$ is given by $K_\alpha(s,t) = \exp\left( - \int_{s}^{t} \a(u) du \right) $ for all $s<t$.

\begin{pro}\label{pro:alpha_proces}
	Consider an interval $T \subset \R$ and a non-negative measurable map $\a : T \to [0,+\infty]$. Then $K_\a : (s,t) \in T \times T \mapsto \exp\left( - \int_{\min(s,t)}^{\max(s, t)} \a(u) du \right)$ is a positive semi-definite kernel and the centered Gaussian measure with covariance function $K_\a$ is Markov. Let  $P_\a \in \PP(\R^T)$ denote this centered Gaussian measure with covariance function $K_\a.$
\end{pro}
\begin{proof}
	First, we establish that $K_\a$ is a positive semi-definite kernel. Fix $s<t \in T^2$ and set $A_{s,t} := \begin{pmatrix}
		K_\a(s,s) & 	K_\a(s,t) \\ 	K_\a(t,s) & 	K_\a(t,t)
	\end{pmatrix}.$ Since $\Tr(A_{s,t}) = K_\a(s,s) + K_\a(t,t) = 1 + 1 = 2 > 0 $ and $\det(A_{s,t}) =  1 - \exp(-2\int^{t}_{s}\a(v)dv) \geq 0$, the matrix $A_{s,t}$ is positive semi-definite and the probability measure $\mu_{s,t} := \NN(0,A_{s,t})$ is well defined. For all $s<t<u \in T^3$, the composition formula of Lemma \ref{lem:compo_gaussien} implies $\mu_{s,t} \cdot \mu_{t,u} = \NN(0,A),$ where \[A = \begin{pmatrix}
		K_\a(s,s) & K_\a(s,t)K_\a(t,u)/K_\a(t,t) \\  K_\a(s,t)K_\a(t,u)/K_\a(t,t) & K_\a(u,u)
	\end{pmatrix}.\] Since 
	\begin{equation}\label{eq:Markov_P_alpha}
		\frac{K_\a(s,t)K_\a(t,u)}{K_\a(t,t)} = \exp\left( -\int_s^t \a(u) du\right) \exp\left( -\int_t^u \a(x) dx\right)/1 = K_\a(s,u),
	\end{equation} it follows that $A = A_{s,u}$, which implies that $\mu_{s,t} \cdot \mu_{t,u} = \mu_{s,u}.$ According to the Kolmogorov extension theorem, there exists a unique Markov measure $Q \in \PP(\R^T)$ such that, for all $s<t \in T^2$, $Q^{s,t} = \mu_{s,t}$. This implies $K_\a$ is the covariance function of $Q$, thus a positive semi-definite kernel. Finally, by Proposition \ref{pro:criteria_gauss_Markov} and Equation \eqref{eq:Markov_P_alpha}, $P_\a$ is a Markov measure.
\end{proof}
The following Remark provides a probabilistic intuition on $P_\alpha$ when $\a$ is constant.
\begin{remq}\label{remq:interpretation_sto_P_alpha}
	\begin{longlist}
		\item If $\a$ is constantly equal to $0$, then for all $s<t \in \R^2$, $K_\a(s,t) = 1$. Thus, $P_\a$ is the law of a completely correlated process $ (Z)_{t \in \R}$, where $Z \sim \NN(0,1).$ 
		\item If $\a$ is constant equal to $+\infty$, then for every $s<t \in \R^2$, $K_\a(s,t) = 0$ . Hence, $P_{\a}$ is the law of a family $ (X_t)_{t \in \R}$ of i.i.d. random variables with law $\NN(0,1).$
		\item  The stationary Ornstein-Uhlenbeck process with parameter $\a \in \R_+^*$ is defined as the solution to the stochastic differential equation 
		\begin{equation}
			\begin{cases}
				dX_t = -\alpha X_t dt + d B_t \\
				X_0 =Z
			\end{cases},
		\end{equation}
		
		where $(B_t)_{t \geq 0}$ is a standard Brownian motion and $Z$ is a random variable independent from $(B_t)_{t \geq 0}$ with law $\NN(0,1/2\a)$. It is well established that a stationary Ornstein-Uhlenbeck process is a centered Gaussian process with covariance function $(s,t) \mapsto \frac{1}{2\alpha}\exp\left(-\a |t-s| \right).$ Thus $P_{\a}$ is the law of $(\sqrt{2\a} X_t)_{t \geq 0}$, that is a rescaled Ornstein-Uhlenbeck process with parameter $\alpha$.
	\end{longlist}
\end{remq}
We can now prove the Proposition.
\begin{proof}[Proof of Proposition \ref{proA:mimicking}]
	Set $\a := \a^K$. To prove \emph{Point \ref{Existence}}, let $P$ denote the law of $X$. According to Proposition \ref{pro:alpha_proces}, the Gaussian measure $Q$ with the same mean function as $P$ and covariance function $K'$ given by Formula \eqref{eq:transfor_formule_cov} is well defined and Markov.  Consider a process $Y = (Y_t)_{t \in \R}$ with law $Q.$ For all $t \in \R$, $\text{Law}(Y_t) = \NN\left(\E(Y_t),K'(t,t)\right) = \NN\left(\E(X_t),K(t,t)\right) = \text{Law}(X_t)$, thus $Y$ satisfies Condition $(a)$.  Let $\theta$ be defined by 
	\begin{equation*}
		\theta : x \in \R \mapsto 
		\begin{cases}
			\frac{e^x - 1 - x}{x} & \text{if } x \neq 0 \\ 0 & \text{if } x = 0
		\end{cases} .
	\end{equation*}
	The map $\theta$ is continuous, and for all $x \in \R$,  $e^x = 1 + x + x \theta(x)$. Hence, for all $v \in \R$, 
	\begin{align*}
		L^{K'}_v(h) &= h^{-1}\left( 1 - c_{K'}(v,v+h) \right)\\
		&= h^{-1} \left( 1 - \exp\left(-\int_v^{v+h} \a(u) du \right) \right)\\ &= h^{-1}\int_v^{v+h} \a(u) du + h^{-1}\left(\int_v^{v+h} \a(u) du \right) \theta\left(-\int_v^{v+h} \a(u) du\right).
	\end{align*}
	To show \eqref{eq:univ_conv}, fix $s<t \in \R^2.$ For all $h \in ]0,1]$, define \[M_{h} := \sup \ens{\left| \a(v) - h^{-1} \int_v^{v+h} \a(u)du\right|}{v \in [s,t]}.\] Since 
	\begin{align*}
		M_{h} &\leq \sup_{v \in [s,t]} h^{-1} \int_v^{v+h} \left| \a(v) - \a(u) \right| du \\
		&\leq \sup \ens{|\a(w)-\a(u)|}{(w,u) \in [s,t+1]^2, |w-u| \leq h}
	\end{align*}
	and $\a$ is uniformly continuous on $[s,t+1]$, $\lim_{h \to 0^+} M_{h} = 0.$ Hence, for all $v \in [s,t] $,
	\begin{align*}
		&\left|\a(v) \right.  \left. -L^{K'}_v(h) \right| \\
		&= \left| \a(v) - h^{-1}\int_v^{v+h} \a(u) du \right. \hspace{0cm}\left. - h^{-1}\left(\int_v^{v+h} \a(u) du \right)  \theta\left(-\int_v^{v+h} \a(u) du\right)  \right| \\
		&\leq \left|  \a(v) - h^{-1}\int_v^{v+h} \a(u) du \right| + \left| h^{-1}\left(\int_v^{v+h} \a(u) du \right) \theta\left(-\int_v^{v+h} \a(u) du\right) \right| \\
		&\leq M_{h} + \left(M_{h} + \sup_{u \in [s,t]} |\a(u) | \right) \sup \ens{|\theta(x)|}{|x| \leq h \left(M_{h} + \sup_{u \in [s,t]} |\a(u) | \right)} \\
		& =: N_{h},
	\end{align*}
	where we used $$\left| h^{-1}\int_v^{v+h} \a(u) du \right| \leq  \left| h^{-1}\int_v^{v+h} \a(u) du - \a(v) \right| + |\a(v)| \leq M_{h} + \sup_{u \in [s,t]} |\a(u)|.$$ Hence,  
	\begin{equation*}
		\sup_{v \in [s,t]} \left| \a^K(v) - L_v^{K'}(h) \right| \leq N_{h}.
	\end{equation*}  
	Since $\lim_{h \to 0^+} N_{h}=0$, this proves Equation \eqref{eq:univ_conv} and in particular Condition $(b).$ Since $Q$ is Markov and $\text{Law}(Y) = Q$, the process $Y$ is Markov, \ie , satisfies Condition $(c)$. To prove \emph{Point \eqref{Point:Uniqueness}}, we consider a Gaussian process $Y$ with covariance function $K'$ satisfying Conditions $(a),(b)$ and $(c)$. According to Condition $(a)$, $Y$ clearly has the same mean function as $X$ and we are left to prove Formula \eqref{eq:transfor_formule_cov}. According to Remark \ref{rq:pos_markov}, for all $s \in \R$, we can define $g_s : t \in \R \mapsto -\log\left(c_{K'}(s,t)\right) \in \R_+$. We have $g_s(s) = 0$ and, for all $(t,h) \in \R \times \R_+^*$, $c_{K'}(s,t+h) = c_{K'}(s,t)c_{K'}(t,t+h)$. By definition of $\a$, it follows that :
	\begin{align*}
		h^{-1} (g_s(t+h)-g_s(t)) &= - h^{-1}\left(\ln(c_{K'}(s,t+h)) - \ln(c_{K'}(s,t))\right) \\
		&=-h^{-1} \ln(c_{K'}(s,t+h)/c_{K'}(s,t)) \\
		&= -h^{-1}\ln(c_{K'}(t,t+h)) \\
		&= -h^{-1} \left(c_{K'}(t,t+h)-1 \right) \left[1+\ee\left(c_{K'}(t,t+h)-1\right)\right]\\
		&\dcv{h \to 0^+} \a(t),
	\end{align*} 
	where $\ee$ is used as defined in Notation \ref{nott:DL_Taylor} and the convergence is obtained using Condition $(c).$ This implies that $g_s(t) = g_s(s) + \int_s^t \a(u)du = \int_s^t \a(u) du$, \ie ,\ Formula \eqref{eq:transfor_formule_cov}. To prove  \emph{Point \eqref{Point:SDE}}, set $$b(t,x) := m'(t) +\left(\frac{\sigma'(t)}{\sigma(t)}-\a(t)\right)(x-m(t)) \quad \text{and} \quad  \sigma(t,x) := \sigma(t) \sqrt{2\a(t)}.$$ We now prove strong uniqueness holds for $(b,\sigma)$. According to \cite[Chapter $5$,Theorem $2.5$]{karatzas_brownian_1988}, it is sufficient to prove that for all $T>0$ and $n \geq 0$, there exists $ K_{T,n} \in \R_+$ such that:  
	\begin{equation}\label{eq:uniqueness_SDE}
		\forall x, y \in [-n,n], \forall t \in [0,T], \left|b(t,x) - b(t,y) \right| + \left|\s(t,x)- \sigma(t,y) \right| \leq K_{T,n} |x-y|.
	\end{equation}
	By continuity of $\a$, the constant $K_{T,n} := \sup_{t \in [0,T]} |\s'(t)/\sigma(t)-\a(t)|$ satisfies \eqref{eq:uniqueness_SDE} and strong uniqueness holds for $(b,\s)$. To prove strong existence, we must establish the existence of  a strong solution for any probability space $(\Omega,\FF,\mathbb{P})$ endowed with a Brownian motion $(B_t)_{t \geq 0}$ and an initial condition $Z$ independent from the Brownian motion. First, we prove the existence of a strong solution for the SDE \begin{equation}\label{eq:SDE}
		\begin{cases*}
			d\ti{Z}_t = \left( - \a (t) \ti{Z}_t \right) dt + \sqrt{2 \a(t)} d B_t \\
			\ti{Z}_0 = Z \sim \NN(0,1)
		\end{cases*}.
	\end{equation}
	According to \cite[Chapter 5, Section 6, Page 354]{karatzas_brownian_1988}, a strong solution to the SDE \eqref{eq:SDE} is the process $(\ti{Z}_t)_{t \geq 0}$ defined by \[\ti{Z}_t := \Phi(t)Z + \int_0^t \frac{\Phi(t)}{\Phi(u)} \sqrt{2 \a(u)} d B_u,\] where $\Phi(t) := \exp\left( - \int_0^t \a(u) du \right).$ Hence, the process $(\ti{Z}_t)_{ t \geq 0}$ is centered Gaussian and its covariance function $K'' : (s,t) \mapsto \E(\ti{Z}_s\ti{Z}_t)$ is computed as follows. For all $s<t \in \R^2$,
	
	\begin{align*}
		K''(s,t) &= \E\left(  \Phi(s) \Phi(t) Z^2 \right) + \E\left( \int_0^s \frac{\Phi(s)}{\Phi(u)} \sqrt{2\a(u)}dB_u \int_0^t \frac{\Phi(t)}{\Phi(v)}  \sqrt{2\a(v)}dB_v\right) \\
		&= \Phi(s)\Phi(t) + \Phi(s) \Phi(t) \int_0^s \frac{2 \a(u)}{\Phi(u)^2} du \\
		&= \Phi(s) \Phi(t) \left( 1 + \int_0^s 2 \a(u) \exp \left( \int_0^u 2 \a(v) dv \right) du \right) \\
		&= \Phi(s)\Phi(t) \left( 1 + \left[ \exp  \left( \int_0^u 2 \a(v) dv \right)\right]_0^s\right) \\
		&= \Phi(s)\Phi(t)\left(1 + \Phi(s)^{-2} -1 \right) = \Phi(t)/\Phi(s) = K_\a(s,t).
	\end{align*}
	Hence $P_\a$ and  $\text{Law}((\ti{Z}_t)_{t \geq 0})$ are both Gaussian measures with identical mean and covariance function and are thus equal. Consider now the process ${Z}_t := m(t) + \sigma(t)\ti{Z}_t.$ This process is Gaussian, has mean function $m$ and covariance function given by $\eqref{eq:transfor_formule_cov}$. To prove that $(Z_t)_{t \geq 0}$ is solution of the SDE \eqref{eq:SDE_complet}, just apply the Itô formula with $f(t,x) = m(t) + \sigma(t)x$ and use that $(\ti{Z}_t)_{t \geq 0}$ is a solution of the SDE \eqref{eq:SDE}.
\end{proof}

\begin{remq}\label{remq:generalization}
	The result remains valid for processes defined on an interval of $\R$ rather than on the whole real line, and the proof easily extends to the case $\alpha^K \equiv +\infty$. In particular, this includes the case of fractional Brownian motion discussed in the introduction.
\end{remq}

\section{Proof of Theorem \ref{them:transform_is_mimicking}}\label{sec:Identification_criteria}

The first subsection proves the finite-dimensional convergence of $P_{[R_n]}$, $n \ge 1$. 
The second establishes a tightness criterion allowing us to upgrade this convergence to weak convergence. 
Finally, the third subsection presents a new example.  

\subsection{Proof of the finite-dimensional convergence}

The proof of the following proposition is straightforward, but the result is nevertheless crucial. This relies on the fact that, for every Gaussian measure $P \in \PP(\R^\R)$, $s<t \in \R^2$ and $R = (t_0 < \dots < t_m)$, the correlation coefficient of $P_{\{R\}}^{s,t}$ equals the product of the correlation coefficients of  $P^{t_0,t_1}, \dots, P^{t_{d-1},t_d}$.
\begin{pro}\label{pro:formulation_calculatoire}
	Fix $(\mu_t)_{t \in T} \in \left(\GG_1^*\right)^T$ and let $P,P' \in \Marg((\mu_t)_{t \in T})$ denote two Gaussian processes with covariance functions $K$ and $K'$, respectively. For all $s<t \in T^2$ and $R_n = (t_k^n)_{k =1,\dots,m_n} \in \SS_{[s,t]}, n \geq 1$, the following conditions are equivalent: 
	\begin{longlist}
		\item $P'^{s,t} = \lim_{n \to +\infty} P_{\{R_n\}}^{s,t}$,
		\item\label{point:conv_fraction} $ K'(s,t) = \lim_{n \to + \infty} \frac{\prod_{k=0}^{m_n-1} K(t_k^n,t_{k+1}^n)}{{\prod_{k=1}^{m_n-1}} K(t_k^n,t_k^n)} .$
	\end{longlist}
\end{pro}
\begin{proof}
	According to Lemma \ref{lem:stab_gaussienne}, $P'^{s,t} = \lim_{n \to + \infty} P_{\{R_n\}}^{s,t}$ if and only if \ $ \Sigma_{P'^{s,t}} = \lim_{n \to + \infty} \Sigma_{P_{\{R_n\}}^{s,t}}$. Hence, writing $u_n$ the fraction appearing in the limit of Point \eqref{point:conv_fraction}, the composition formula of Lemma \ref{lem:compo_gaussien} implies that $P'^{s,t} = \lim_{n \to + \infty} P_{\{R_n\}}^{s,t}$ amounts to 
	\[\begin{pmatrix}
		\Sigma_{\mu_s} & K'(s,t) \\
		K'(s,t) & \Sigma_{\mu_t} \\
	\end{pmatrix} = \lim_{n \to + \infty}
	\begin{pmatrix}
		\Sigma_{\mu_s} & u_n \\
		u_n & \Sigma_{\mu_t} \\
	\end{pmatrix}
	,\] that is $K'(s,t) = \lim_{n \to + \infty} u_n .$
\end{proof}

If $c_K :(s,t) \mapsto K(s,t)/\sqrt{K(s,s)K(t,t)}$ denotes the correlation function of $P$ and $c_{K'} :(s,t) \mapsto K'(s,t)/\sqrt{K'(s,s)K'(t,t)}$ the correlation function of $K'$, the second point becomes \[\lim_{n \to + \infty} c_K(t_0^n,t_1^n)\cdots c_K(t_{m_n-1}^n,t_{m_n}^n) = c_{K'}(s,t).\]

We can now prove our convergence result. Since the expansion of $\ln(1+x)$ at point $0$ will often be used, we fix a notation.
\begin{nott}\label{nott:DL_Taylor}
	Let $\ee : ]-1,+\infty[ \to \R$ be the function defined by \[ \ee(x) = \begin{cases}
		\frac{\ln(1+x)-x}{x} &\text{if } x \in ]-1,+\infty[ \setminus \{0\} \\
		0 &\text{if } x = 0
	\end{cases}.\] This function is continuous, $\ee(0) = 0$, and for every $x \in ]-1,+ \infty[$, $\ln(1+x) = x + x \ee(x)$.
\end{nott}

\begin{lem}\label{them:Markov_stationnaire}
	Consider an interval  $T \subset \R$  and a continuous positive semi-definite kernel $K : T \times T \to \R$ with constant variance function equal to $1.$ Let $P$ denote the centered Gaussian process with covariance function $K.$ For all $(s,t,h^*) \in T \times T \times \R_+^*$, define  $$C_{s,t}(h^*) := \ens{(v,h)}{v \in [s,t], v +h \in [s,t], h \in ] 0, h^*]}.$$
	
	\begin{longlist}
		\item If $\a^K$ is well defined, continuous, and for all $ s<t \in T \times T$,
		\begin{equation}\label{eq:conv_fini}
			\sup_{(v,h) \in C_{s,t}(h^*)} \left|L^K_v(h) - \a^K(v)\right| \dcv{h^* \to 0^+} 0,
		\end{equation} 
		then $P$  admits a strong local Markov transform, and this strong Markov is $P_{\a^K}$.
		\item If, for all $s<t  \in T \times T$, 
		\begin{equation}\label{eq:conv_infini}
			\inf_{(v,h) \in C_{s,t}(h^*)} L^K_v(h) \dcv{h^* \to 0^+} + \infty.
		\end{equation}
		Then $P$  admits a strong local Markov transform, and this strong local Markov transform is $P_{+\infty}$.
	\end{longlist}
\end{lem}
\begin{proof}
	In order to prove the \emph{first point}, fix $s<t \in T^2$ and $(R_n)_{n \geq 1} \in \left(\SS_{[s,t]}\right)^{\N^*}$ such that $ \lim_{n \to +\infty} \s_{R_n} =  0.$ Set $L := L^K $, $\a := \a^K$, $R_n = (t_0^n,\dots,t_{m_n}^n)$ for $n \geq 1$ and  $h^n_k := t_{k+1}^n - t_k^n$ for  $k = 0,\dots, m_n-1$. According to Proposition \ref{pro:formulation_calculatoire}, we must establish the limit \[\prod_{k=0}^{m_n-1} K(t_k^n,t_{k+1}^n) \dcv{n \to +\infty} \exp\left(-\int_s^t\a(u)du\right),\] that is,
	\begin{equation}\label{eq:conv_num_schema}
		\sum_{k=0}^{m_n-1} \ln\left(K(t_{k}^n,t_{k}^n + h_k^n)\right) \dcv{n \to + \infty} -\int_s^t \a(u) du.
	\end{equation} 
	Using $\ee$ as defined in Notation \ref{nott:DL_Taylor}, for all $(v,h) \in T \times \R^*$ such that $v+h \in T$ and $K(v,v+h)>0$ , we get $\ln(K(v,v+h)) = (K(v,v+h)-1)\left[1+\ee\left(K(v,v+h)-1\right)\right] = -h L_v(h)\left[1+\ee\left(-h L_v(h)\right)\right].$ Thus, \begin{equation}\label{eq:dl}
		\sum_{k=0}^{m_n-1} \ln\left(K(t_{k}^n,t_{k}^n + h_k^n)\right) =  -\sum_{k=0}^{m_n-1} h_k^n L_{t_k^n}(h_k^n)\big[1 + \ee\left(-h_k^n L_{t_k^n}(h_k^n)\right)\big],
	\end{equation} which implies that $u_n := \Big{|} \underbrace{\sum_{k=0}^{m_n-1} \ln\left(K(t_{k}^n,t_{k}^n + h_k^n)\right)}_{\leq 0} - \underbrace{\left(-\int_s^t \a(u) du \right)}_{\leq 0} \Big{|}$ satisfies
	\begin{align*}
		u_n &= \left| \sum_{k=0}^{m_n-1} -h_k^n L_{t_k^n}(h_k^n)\left[1 + \ee\left( -h_k^nL_{t_k^n}(h_k^n) \right) \right] - \left( - \int_s^t \a(u) du \right)   \right| \\
		&= \left| \sum_{k=0}^{m_n-1} h_k^n L_{t_k^n}(h_k^n) - \left(  \int_s^t \a(u) du \right) \right. + \left. \sum_{k=0}^{m_n-1} h_k^n L_{t_k^n}(h_k^n) \ee\left(-h_k^n L_{t_k^n}(h_k^n)  \right) \right|\\
		&= \left| \left( \sum_{k=0}^{m_n-1} h_k^n L_{t_k^n}(h_k^n) - \sum_{k=0}^{m_n-1 } h_k^n \a(t_k^n)\right)  +  \left( \sum_{k=0}^{m_n-1} h_k^n \a(t_k^n) - \int_s^t \a(u)du \right) \right.  \\
		&\hspace{2cm} + \left. \sum_{k=0}^{m_n-1} h_k^n L_{t_k^n}(h_k^n)\ee\left(-h_k^nL_{t_k^n}(h_k^n) \right) \right| \\
		&\leq  \underbrace{\left|\sum_{k=0}^{m_n-1} h_k^n L_{t_k^n}(h_k^n) - \sum_{k=0}^{m_n-1} h_k^n \a(t_k^n) \right|}_{:= \ a_n} +\underbrace{\left| \sum_{k=0}^{m_n-1} h_k^n \a(t_k^n) - \int_s^t \a(u) du \right|}_{:= \ b_n}  \\
		&\hspace{2cm }+ \underbrace{\left|\sum_{k=0}^{m_n-1} h_k^n L_{t_k^n}(h_k^n) \ee\left(-h_k^n L_{t_k^n}(h_k^n)\right) \right|}_{:= \ c_n}  \\
	\end{align*}
	It remains to prove that $\lim_{n \to + \infty} a_n = \lim_{n \to + \infty} b_n = \lim_{n \to + \infty} c_n = 0.$ For all $n \geq 1$,
	\begin{equation*}\label{eq:maj_a_n}
		0 \leq a_n \leq \sum_{k = 0}^{m_n- 1} h_k^n \left| L_{t_k^n}(h_k^n) - \a(t_k^n) \right| \leq |t-s| \sup_{(v,h) \in C_{s,t}(\s_{R_n})} |L_v(h) - \a(v) |.
	\end{equation*}
	According to Hypothesis \eqref{eq:conv_fini}, $\lim_{n \to + \infty} a_n = 0.$  Since $\a$ is continuous and continuous functions are Riemann-integrable, we immediately get $\lim_{n \to + \infty} b_n = 0.$ In order to prove that $\lim_{n \to + \infty} c_n = 0$, set $C_{s,t} := \ens{(v,h)}{v \in [s,t], v +h \in [s,t],  0 \leq h}$ and  $L_v(0) := \a(v)$ for all $v \in [s,t]$. We want to prove that $L$ is continuous on the compact set $C_{s,t}.$ Since $K$ is continuous, we know that $L$ is continuous on $C_{s,t} \setminus \left([s,t] \times \{0\}\right)$ and we are left proving the continuity on $[s,t] \times \{0\}$. Fix $ \ee >0$, $v^* \in [s,t]$ and consider $ \a_\ee \in \R_+^*$ such that \[ \begin{cases}
		\sup_{ (v,h) \in C_{s,t}(\a_\ee)} |L_v(h) - \a(v)| \leq \ee/2\\ \sup_{v \in [s,t] ; |v-v^*| \leq \a_\ee} |\a(v)-\a(v^*)| \leq \ee/2
	\end{cases}. \] For $(v,h) \in C_{s,t} \cap \left( [v^*-\a_\ee, v^* + \a_\ee] \times [-\a_\ee, \a_\ee] \right) \subset C_{s,t}(\a_\ee) \cup \left([v^*-\a_\ee,v^*+\a_\ee] \times \{0\}\right)$,
	\begin{align*}
		|L_v(h) - L_{v^*}(0) | &= |L_v(h) - \a(v^*) |  \\ &\leq |L_v(h) - \a(v)| + |\a(v) - \a(v^*)| \\ &\leq \sup_{ (v,h) \in C_{s,t}(\a_\ee)} |L_v(h) - \a(v)| + \sup_{v \in [s,t];|v-v^*| \leq \a_\ee} |\a(v)-\a(v^*)| \\ &\leq \ee/2 + \ee/2 = \ee,
	\end{align*}
	which shows the desired continuity.
	Hence $M := \sup_{(v,h) \in C_{s,t}} L_v(h)$ is finite and \begin{equation*}
		0 \leq c_n \leq \sum_{k=0}^{m_n-1} h_k^n L_{t_k^n}(h_k^n) \left| \ee\left(- h_k^n L_{t_k^n}(h_k^n)\right) \right| \leq |t-s| M \sup_{ |x| \leq M \s_{R_n}} |\ee(x)|,
	\end{equation*}
	which implies that $\lim_{n \to + \infty} c_n = 0$, concluding the proof of the first point. We now prove the \emph{second point}. Using the same notation as before, we are left to prove \eqref{eq:conv_num_schema} with $\a$ constant equal to $+ \infty$, that is $\lim_{n \to +\infty} \sum_{k=0}^{m_n-1} \ln\left({K}(t^n_{k},t^n_{k+1})\right) = - \infty$. As for \eqref{eq:dl}, this amounts to show \[ \sum_{k=0}^{m_n-1}  h_k^n L_{t_k^n}(h_k^n)\big[1 + \ee\left(-h_k^n L_{t_k^n}(h_k^n)\right)\big] \dcv{n \to + \infty} + \infty. \]
	One can find a rank $N \geq 1$ such that for all $n \geq N$ and $k = 0,\dots m_n-1$, we have $ \ee\left(-h_k^n L_{t_k^n}(h_k^n)\right) = \ee\left(K(t_k^n,t_k^n + h_k^n)-1\right) \geq -1/2$. Thus, for all $n \geq N$,  
	\begin{align*}
		\sum_{k=0}^{m_n-1} h_k^n L_{t_k^n}(h_k^n)\left[1 + \ee \left(-h_k^n L_{t_k^n}(h_k^n)\right)\right] &\geq \sum_{k=0}^{m_n-1} h_k^n L_{t_k^n}(h_k^n)2^{-1} \geq \frac{t-s}{2}v_n,
	\end{align*} 
	where $v_n :=\inf_{(v,h) \in C_{s,t}(\s_{R_n})} L_v(h).$ According to Hypothesis \eqref{eq:conv_infini}, we get $\lim_{n \to + \infty} v_n = + \infty$, which finishes the proof.  
\end{proof}

By Remarks \ref{remq:centered_enough} and \ref{remq:finite_dim_conv_gaussian}, we have the following result, which corresponds to Theorem \ref{them:transform_is_mimicking} with finite-dimensional convergence instead of weak convergence.
\begin{cor}\label{cor:fd_conv}
	Let $T \subset \R$ be an interval and $P \in \PP(\R^T)$ be a Gaussian process with positive covariance. 
	\begin{enumerate}
		\item If $\a^K$ is well defined, continuous, and for all $ s<t \in T \times T$,
		\begin{equation*}
			\sup_{(v,h) \in C_{s,t}(h^*)} \left|L^K_v(h) - \a^K(v)\right| \dcv{h^* \to 0^+} 0,
		\end{equation*} 
		then $P$  admits a strong global Markov transform, and this strong Markov is the process $P'$ with same mean function than $P$ and covariance $K':(s,t) \mapsto  \sqrt{K(s,s)K(t,t)}K_{\a^K}(s,t
		).$
		\item If, for all $s<t  \in T \times T$, 
		\begin{equation}\label{eq:conv_infini_bis}
			\inf_{(v,h) \in C_{s,t}(h^*)} L^K_v(h) \dcv{h^* \to 0^+} + \infty,
		\end{equation}
		then $P$  admits a strong local Markov transform, and this strong local Markov transform is the Gaussian process with same mean than $P$ and covariance $$K':(s,t) \mapsto  \sqrt{K(s,s)K(t,t)}K_{\a^K}(s,t
		) = K(t,t)\1_{s=t}.$$
	\end{enumerate}
\end{cor}

\subsection{Proof of the weak convergence}

The goal of this section is to upgrade the finite-dimensional convergence of Corollary \ref{cor:fd_conv} to weak convergence. Recall that, given $a< b \in \R^2$, we can endow $\CC([a,b],\R)$ with the topology of uniform convergence, inducing a topology on $\PP(\CC([a,b],\R))$, which we simply call the weak topology. We shall write $ P_n \xrightarrow[n\to +\infty]{w} P$ to express that a measures $P_n \in \PP(\CC([a,b],\R), n \geq 1$ converge to $P \in \PP(\CC([a,b],\R))$ with respect to this topology, \ie ,\ $\int f d P_n \dcv{ n \to + \infty} \int f d P$ for every continuous bounded function $f : (\CC([a,b], \R), \n{ \cdot}{\infty} ) \to (\R,|\cdot|).$ Weak convergence implies convergence with respect to the finite-dimensional topology, but in general the converse is false. However, if the family $(P_n)_{n \geq 1}$ is \emph{tight}\footnote{In this paper, we only use tightness criterion, not the definition of tightness. For completeness, we refer to \cite{billingsley_probability_1995}.}, then both convergences are equivalent. To carry out our convergence result from the finite-dimensional topology to the weak convergence topology, we need the Kolmogorov-Chentsov criterion to ensure that our measures $P$ and $P_{[R_n]}$ are concentrated on $\CC([a,b],\R)$ and the Kolmogorov tightness criterion to ensure that  $\{P\} \cup \{P_{[R_n]} \ ; \ n \geq 1\}$ is tight. For a proof of the Kolmogorov continuity criterion, we refer to \cite[Theorem 2.9]{le_gall_brownian_2016}, while for the  Kolmogorov tightness criterion we refer to \cite[Theorem 23.7]{kallenberg_foundations_2021}.
\begin{them}\label{them:tighness_criteria}
	Let  $T \subset \R$ be compact. 
	\begin{description}
		\leftskip=0.3in
		\item[Kolmogorov-Chentsov criterion:] If $P \in \PP(\R^T)$ and \ \[\exists a,b,C \in \R_+^*, \forall (s,t) \in T^2, \int_{\R^2} |x-y|^b P^{s,t}(dx,dy) \leq C |t-s|^{1+a},\] then $P$ is supported on continuous paths.
		\item[Kolmogorov tightness criterion:] If $(P_n)_{n \geq 1} \in \PP(\CC(T,\R))^{\N^*}$ and \ \[\exists a,b,C \in \R_+^*, \forall (s,t) \in T^2, \sup_{n \geq 1} \int_{\R^2} |x-y|^b P_n^{s,t}(dx,dy) \leq C |t-s|^{1+a},\] then $(P_n)_{n \geq 1}$ is tight with respect to weak convergence.
	\end{description}
	
\end{them}

To prove the uniform bounds in Theorem  \ref{them:tighness_criteria}, we first establish a technical lemma.

\begin{lem}\label{lem:inequality_mark_variance}
	Consider a continuous positive semi-definite  kernel $K : \R^2 \to \R$ such that, for all $t \in \R$, $K(t,t) = 1$ . Fix $a<b \in \R^2$ and let $P$ denote a centered Gaussian measure with covariance function $K.$ Assume that, for all $s<t \in \R^2$, 
	\begin{equation}\label{eq:hyp_conv_unif_R_mesure}
		\sup_{v \in [s,t]} \left|L_v^K(h) - \a^K(v) \right| \dcv{h \to 0^+} 0.
	\end{equation} 
	Then \[ \forall (R_n)_{n \geq 1} \in \A, \exists M \in \R_+, \forall (s,t) \in [a,b]^2,\forall n \in \N^*, 1- K_n(s,t) \leq M|s-t|,\] where $K_n$ stands for the covariance function of $P_{[R_n]}.$
\end{lem}

\begin{proof}
	Fix $(R_n)_{n \geq 1} \in \A$ and $s<t \in [a,b]^2$. Define $L_t := L_t^K$, $\a := \a^K$ and, for all $v \in \R,$ $L_v(0) :=\a(v)$.  Set $\sigma^* := \sup_{n \geq 1} \s_{R_n}$,  $ M_1:= \sup_{0 \leq h \leq \sigma^*, v \in [a,b]} |L_v(h) - \a(v)|$, $M_2 := \sup_{x,y \in [a,b]} |\a(x)-\a(y) |$, $M_3 := \sup_{0 \leq h \leq \sigma^*, v \in [a,b]} L_v(h)$,  $M_4 := \sup_{|x| \leq \sigma^* M_3} |\ee(x) |$ and $M_5 := M_1 + M_2 + M_3 M_4$, where $\ee$ is defined in Notation \ref{nott:DL_Taylor}. As in Lemma \ref{them:Markov_stationnaire},  the Assumption \eqref{eq:hyp_conv_unif_R_mesure} implies $L$ that is continuous on $[a,b] \times [0,\s^*]$ and thus $M_3< + \infty$. Since $K$, $\a$, and $\ee$ are continuous, $M_5 < +\infty$. For all $n \geq 1$, we write $\{s,t\} \cup \left( R_n \cap ]s,t[ \right) = \{t_0 < \cdots < t_{m}\}$, and for all $k \in \{0, \dots , m-1\}$, $h_k := t_{k+1} - t_k$. Since $\left(P_{[R_n]}\right)^{s,t} = P_{\{\{s,t\} \cup \left( R_n \cap ]s,t[ \right)\}}^{s,t}$, the composition formula of Lemma \ref{lem:compo_gaussien} gives $K_n(s,t) = K(t_0,t_1) \cdots K(t_{m-1},t_m).$ As in the proof of Lemma \ref{them:Markov_stationnaire},
	\begin{align*}
		\left| \ln(K_n(s,t)) -\right. &\left. \left(-\int_s^t \a(u) du \right) \right| 
		\leq  \left|\sum_{k=0}^{m-1} h_k L_{t_k}(h_k) - \sum_{k=0}^{m-1} h_k \a(t_k) \right|   \\ &\hspace{1cm}+\left| \sum_{k=0}^{m-1} h_k \a(t_k) - \int_s^t \a(u) du \right|  + \left|\sum_{k=0}^{m_n-1} h_k L_{t_k}(h_k) \ee( -h_k L_{t_k}(h_k)) \right| \\
		&\leq   \sum_{k=0}^{m-1} h_k^n \left|L_{t_k}(h_k) - \a(t_k) \right| + \sum_{k=0}^{m-1} \int_{t_k}^{t_{k+1}} |\a(t_k) - \a(u)|du \\ &\hspace{1cm}+  \sum_{k=0}^{m-1} h_k L_{t_k}(h_k)\left|\ee\left(- h_kL_{t_k}(h_k) \right)\right| \\
		&\leq |t-s| M_1 + |t-s| M_2+ |t-s|M_3 M_4 \\
		&= |t-s| M_5.
	\end{align*}
	Thus, defining $M := M_5 + \sup_{u \in [a,b]} |\a(u)|$, we obtain $\ln(K_n(s,t)) \geq   - M_5|t-s| -\int_s^t \a(u) du \geq  - M|t-s|$. Hence, for all $s<t \in [a,b]^2$, $K_n(s,t) \geq \exp(-M|t-s|) \geq 1-M|s-t|$, which proves our result.
\end{proof}
We can now prove Theorem \ref{them:transform_is_mimicking}.
\begin{proof}[Proof of Theorem \ref{them:transform_is_mimicking}]
	Let $P$ denote the centered Gaussian measure with covariance function $K.$  Assume that, for all $s<t \in \R^2$,
	\begin{equation*}
		\sup_{v \in [s,t]} \left|L^K_v(h) - \a^K(v)\right| \dcv{h \to 0^+} 0.
	\end{equation*}  
	Then  \[ \forall (R_n)_{n \geq 1} \in \A,  \forall a<b \in \R^2, \proj^{[a,b]}_\# P_{[R_n]} \xrightarrow[n \to + \infty]{w} \proj^{[a,b]}_\# P_{\a^K}.\]
	By Lemma \ref{lem:inequality_mark_variance}, there exists $ M \in \R_+$  such that for  every $ (s,t,n) \in [a,b]^2 \times \N^*$, we have $1-K_n(s,t) \leq  M|t-s|.$
	Hence, \[\int_{\R^2} |x-y|^2 dP_{[R_n]}^{s,t}(x,y) = K_n(s,s) + K_n(t,t) - 2 K_n(s,t) = 2(1-K_n(s,t)) \leq 2M\cdot |s-t|.\]
	Since for a centered Gaussian random variable $X$, one has $\E(X^4) = 3\E(X^2)^2$,
	\begin{equation}\label{eq:inequality_tight_mark}
		\int_{\R^2} |x-y|^4 dP_{[R_n]}^{s,t}(x,y) = 3 \left( \int_{\R^2} |x-y|^2 dP_{[R_n]}^{s,t}(x,y)\right)^2 \leq 12M^2|s-t|^2.
	\end{equation} As $\lim_{n \to +\infty} P_{[R_n]}^{s,t} =  P'^{s,t}$ and $(x,y) \mapsto |x-y|^4$ is lower semicontinuous and bounded by below, the Portmanteau theorem gives $\int_{\R^2} |x-y|^4 dP^{s,t}(x,y) \leq \liminf_{n \to + \infty} \int_{\R^2} |x-y|^4 dP_{[R_n]}^{s,t}(x,y) \leq  12M^2|s-t|^2$. According to the Kolmogorov-Chentsov criterion, the measure of the set $\{\proj_{\#}^{[a,b]}P'\} \cup \{\proj_{\#}^{[a,b]}P_n\}_{n \geq 1} $ are supported on continuous paths. By Inequality \eqref{eq:inequality_tight_mark}, the Kolmogorov tightness criterion applies and   $(\proj^{[a,b]}_\#P_{[R_n]})_{n \geq 1}$ is tight. According to the first point  of Proposition \ref{pro:cov_proc_Markovinifie} and Lemma \ref{them:Markov_stationnaire}, we have $\proj_{\#}^{[a,b]} P_{[R_n]} \xrightarrow[n \to +\infty]{\fd} \proj_\#^{[a,b]} P_{\a^K}$. Combined with tightness, this finally implies  $\proj_{\#}^{[a,b]} P_{[R_n]} \xrightarrow[n \to +\infty]{w} \proj_\#^{[a,b]} P_{\a^K}$.\\
	\textbf{General case:} Assume now $P$ is non-centered with mean function $m$ and positive variance function. For all $a< b$, the map $f : (x_t)_{t \in [a,b]} \mapsto \left(\sqrt{K(t,t)}x_t + m(t) \right)_{t \in [a,b]}$ is $\left(\sup_{t \in [a,b]}\sqrt{K(t,t)}\right)$-Lipschitz, hence continuous with respect to the topology of the uniform norm. Hence, one can apply the centred result to the normalized measure $f^{-1}_\# P$ and push forward the obtained convergence results by $f$, which is continuous\footnote{Note that $L^K$ and $\a^K$ are invariant under renormalization, \ie ,\ $L^{c_K} = L^K$ and $\a^{c_K} = \a^K $. }. At the end, we obtain \[ \forall (R_n)_{n \geq 1} \in \A,  \forall a<b \in \R^2, \proj^{[a,b]}_\# P_{[R_n]} \xrightarrow[n \to + \infty]{w} \proj^{[a,b]}_\# Q,\] where $Q \in \PP(\R^\R)$ is the Gaussian measure with mean function $m$ and covariance function \[K' :(s,t) \mapsto \sqrt{K(s,s)K(t,t)}\exp\left(-\int_s^t \a^K(u) du \right),\] which concludes.
\end{proof}

\begin{remq}
		Note that, unlike in Corollary \ref{cor:fd_conv}, the case $\alpha^K \equiv +\infty$ is excluded in Theorem \ref{them:transform_is_mimicking} in order to obtain weak convergence. This is natural, since the process $P_\infty$, which arises as the limit of the $P_{[R_n]}$ in the sense of finite-dimensional convergence, is the law of a continuum of independent random variables and therefore has discontinuous trajectories.
\end{remq}

\subsection{A new example}

We have already applied this Lemma to compute the Markov transform of fractional Brownian motion. We here give another example.

\begin{pro}\label{pro:mark_bruit_gaussien}
	Consider two intervals $T \subset \R$, $J \subset \R_+^*$, a Brownian motion $(B_u)_{u \geq 0}$ on a probability space $(\Omega,\mathcal{F},\mathbb{P})$ and a family $(k_t)_{t \in T}$ of elements of $L^2(J)$ such that, for every $t \in T$, $\int_{J} k_t(u)^2du = 1$. For every $t \in T$, define $X_t := \int_J k_t(u) dB_u$.
	Then, the law $P \in \PP(\R^T)$ of $(X_t)_{t \in T}$ is a Gaussian measure with covariance function $K : (s,t) \in T \times T \mapsto \int_J k_s(u)k_t(u) du$. Moreover, for every $(s,t,h^*) \in T \times T \times \R_+^*$, define $C_{s,t}(h^*) := \ens{(v,h) \in [s,t[ \times ]0,h^*]}{v+h \in [s,t]}$  and assume:
	\vspace{0.2cm}
	\begin{description}
		\leftskip=0.in
		\item[$(H1)$ ]  $\exists f : T \times J \to \R, \forall (s,t,u) \in T  \times T \times J$,  
		\begin{equation*}
			\sup_{(v,h) \in C_{s,t}(h^*)} |h^{-1}(k_v(u)- k_{v+h}(u)) - f(v,u)| \dcv{ h^* \to 0^+} 0 \ ;
		\end{equation*} 
		\item[$(H2)$]   $\exists (g,h_0) \in L^1(J) \times \R_+^*, \forall h \in ]0,h_0], \forall (t,u),\left| h^{-1}(k_{t}(u)-k_{t+h}(u)) - f(t,u) \right| \leq g(u)$;
		\item[$(H3)$ ]  $\a : t \in T \mapsto \int_J k_t(u)f(t,u)du$ is continuous.
	\end{description}
	\leftskip=0in
	Then $P$ admits a strong Markov transform, which is $P_\a.$
\end{pro}

\begin{proof}
	The hypothesis $k_s \in L^2(J)$ ensures that $(X_t)_{t \in T}$ is well defined. The process $P$ is clearly centered, Gaussian, and for all $(s,t) \in T^2$, 
	\begin{align*}
		K(s,t) = \E\left(\int_J k_s(u)dB_u \int_J k_t(u)dB_u\right) = \int_J k_s(u)k_t(u) du.
	\end{align*} Moreover, for every  $(v,h) \in C_{s,t}(h^*)$, $L^K_v(h) = h^{-1}\left(1-K(v,v+h) \right) = \int_J h^{-1} (k_v(u) - k_{v+h}(u))k_v(u) du.$ Hence,   
	\begin{align*}
		|L^K_v(h)-&\a(v)| = \left| \int_J h^{-1}\left(k_v(u)-k_{v+h}(u)\right) k_v(u) du - \int_J f(v,u) k_v(u)du \right| \\
		&\leq \int_J  \left|  h^{-1}(k_v(u)-k_{v+h}(u)) -  f_v(u)\right| |k_v(u) |du \\
		&\leq  \left( \sup_{w \in [s,t]}  |k_w(u) |\right) \int_J \sup_{(w,h) \in C_{s,t}(h^*)} \left\{ \left| h^{-1}(k_w(u)-k_{w+h}(u)) - f(w,u) \right| \right\}   du,
	\end{align*} 
	which implies that
	\begin{equation}\label{eq:majoration_Lebesgue}
		\begin{aligned}
			&\sup_{(v,h) \in C_{s,t}(h^*)} |L^K_v(h)-\a(v)| \\ &\leq \left( \sup_{w \in [s,t]}  |k_w(u) | \right)\int_J \sup_{(w,h) \in C_{s,t}(h^*)} \left\{ \left| h^{-1}(k_w(u)-k_{w+h}(u)) - f(w,u) \right| \right\}  du.
		\end{aligned}	
	\end{equation} 
	According to $(H_1)$ and $(H_2)$, the Lebesgue's dominated convergence theorem applies to the right-hand side of Equation \eqref{eq:majoration_Lebesgue} which implies that $\lim_{h^* \to 0^+} \sup_{(v,h) \in C{s,t}(h^*)} |L^K_v(h)-\a(v)| = 0$. According to $(H_3)$, $\a$ is continuous and we have $K(t,t) = \int_{J} k_t(u)^2 du =1$. Hence, Lemma \ref{them:Markov_stationnaire} applies, which proves the result.
\end{proof}

\begin{ex}
	Let us apply Proposition \ref{pro:mark_bruit_gaussien} to the law $P$ of the process $(X_s)_{s > 0} := \left( \int_0^{+\infty} \sqrt{s}e^{-\frac{su}{2}} dB_u \right)_{s>0}$. In order to satisfy $(H_2)$, we will rather work on $(X_t)_{t \in [a,b]}$ for a fixed pair $a<b \in \left(\R_+^* \right)^2.$  Fix $s<t \in [a,b]^2$, and for all $(v,u) \in [s,t] \times \R_+^*$, define $k_v(u) := \sqrt{v} e^{-vu/2}$. By applying the Taylor expansion of the exponential map and of $x \in ]-1,+\infty[ \mapsto \sqrt{1+x}$, we get
	\begin{align*}
		h^{-1}(k_v(u)-k_{v+h}(u)) &= k_v(u) \left[h^{-1} \left( 1- \sqrt{1+\frac{h}{v}} e^{\frac{-hu}{2}}\right) \right]\\
		&= \frac{1}{2} k_v(u)\left( u- \frac{1}{v} +\theta(h) \right),
	\end{align*}  where $\theta : \R \to \R$ satisfies $\lim_{h \to 0^+} \theta(h) = 0.$ Setting $f(v,u) := \frac{1}{2}k_v(u) \left(u-1/v \right) $, we get 
	\begin{equation*}
		\left|  h^{-1}(k_v(u)-k_{v+h}(u)) - f(v,u) \right| = \left|\frac{1}{2} k_v(u) \theta(h)\right| \leq \frac{\sqrt{t}}{2} |\theta(h)|,
	\end{equation*}
	for all $(v,u,h) \in [s,t] \times \R_+ \times \R_+^*$. Hence, Hypothesis $(H_1)$ is satisfied. Let $h_0 \in \R_+^*$ be such that $\sup_{ h \in ]0,h_0]}{|\theta(h)|} \leq 1.$ For all $h \in ]0,h_0]$ and $(v,u) \in [a,b] \times \R_+^*$, we have \begin{equation*}
		\left|h^{-1}(k_v(u)-k_{v+h}(u))-f(v,u)  \right| = \left|\frac{1}{2} k_v(u) \theta(h)\right| \leq \frac{1}{2}\sqrt{b} e^{-\frac{au}{2}}=: g(u),
	\end{equation*}
	so that $(H_2)$ holds. Since $f(v,u)k_v(u) =  \frac{1}{2}(u ve^{-vu}-e^{-vu} )$, it follows that for all $v \in [a,b]$,  $\a(v) = \int_0^{+\infty} f(v,u) k_v(u) du = \int_{0}^{+\infty} (uve^{-vu} - e^{-vu} ) du =  0$ . Since $(H_3)$ is obviously true, Proposition \ref{pro:mark_bruit_gaussien} applies and the Markov transform of the law of $(X_t)_{t \in [a,b]}$ is the law of the completely correlated process $(Z)_{t \in [a,b]}$, where $Z \sim \NN(0,1).$ Since this holds for all $a<b \in \left(\R_+^* \right)^2$, the Markov transform of the law of $(X_t)_{t>0}$ is  $(Z)_{t >0}$. By Remark \ref{remq:centered_enough}, the strong local Markov transform of the law of $\left( \int_0^{+ \infty} e^{-tu}dB_u \right)_{t>0}$ is the law of $\left(Z/\sqrt{2t} \right)_{t>0}.$
\end{ex}

\section{Proof of Proposition \ref{proA:existence}}\label{seq:compo_gaussian} In this short section, we prove that any non-singular Gaussian measure admits a weak local Markov transform. Unlike in the other section, we work here in dimension $d\ge 1$. The idea here is to adapt an existence result by Boubel and Juillet in the context of measures with increasing kernels. We denote by $\leq_{st}$ the stochastic order on $\PP(\R)$, namely $\mu \leq_{st} \nu$ if, for every non-decreasing bounded function $f: \R \to \R$, $\int_{\R} f d\mu \leq \int_{\R} f d\nu.$ 
\begin{defi}\label{def:increasing_kernel}
	\begin{longlist}
		\item Consider  $\mu, \nu \in \PP(\R)$ and $P \in \Marg(\mu,\nu).$ We say that $P$ has increasing kernels in the stochastic order if there exists a disintegration  $P(dx,dy) = \mu(dx)k_x(dy)$ and a Borel set $\Gamma$ with $\mu(\Gamma)=1$ such that, for all $x < y \in \Gamma^2$,  $k_x \leq_{st} k_y$ .
		\item Consider  $T \subset \R$  and $P \in \PP\left(\R^{T}\right)$. We say that $P$ has increasing kernels in the stochastic order if, for all $s<t \in T^2$, $P^{s,t}$ has increasing kernels in the stochastic order.
	\end{longlist} 
\end{defi}

For more information about the stochastic order and other orders on probability spaces, we direct the reader to the monographs \cite{muller_comparison_2002,shaked_stochastic_2007}. For additional information about increasing kernels, we refer to  \cite[Proposition/Definition $3.11.$]{boubel_markov-quantile_2022}.

\begin{them}[Boubel--Juillet]\label{them:fonda_NB}
	Consider an interval $T \subset \R$ and $P \in \PP\left(\R^{T}\right)$. If $P$ has increasing kernels in the stochastic order, then $P$ admits a weak local Markov transform.
\end{them}
Theorem \ref{them:fonda_NB} was proved and written with $T = \R$ in \cite{boubel_markov-quantile_2022} , but a simple modification shows that it holds for any interval $T \subset \R$. To adapt their proof, we first need to prove a continuity result for concatenation and to show that a weak local Markov transform of a Gaussian measure is still a Gaussian measure. Its proof mainly relies on the concatenation formula (Lemma \ref{lem:compo_gaussien}) and a standard characterization result for convergence of Gaussian measures (Lemma \ref{lem:stab_gaussienne}).
\begin{lem}\label{lem:compo_gaussien_stab}
	Consider measures $\mu_i \in \GG_{d_i}^*$, $i = 1,\dots,p$, and Gaussian transport plans $P_i^n \in \GG(\mu_i,\mu_{i+1})$, $n \geq 0$ and $i = 1, \dots, p-1$.   If $P^n_i \to P^0_i$ for all $i = 1, \dots, p$, then  $$P^n_1 \circ \cdots \circ P^n_p \to P^0_1 \circ \cdots \circ P^0_p \qquad \text{and} \qquad P^n_1 \cdot~ \dots~ \cdot P^n_p \to P^0_1 \cdot~ \dots~ \cdot P^0_p.$$
\end{lem}
\begin{proof}
	According to Lemma \ref{lem:stab_gaussienne}, for all $i<j \in \{1,\dots, d-1\}^2$, the function $A_{i,j} : (P_1, \dots, P_d) \in \GG(\mu_1,\mu_2) \times \dots \times \GG(\mu_{p-1},\mu_p) \mapsto \Sigma_{P_i}^{d_i,d_{i+1}} \Sigma_{\mu_i}^{-1} \dots  \Sigma_{\mu_{j-1}}^{-1} \Sigma_{P_{j-1}}^{d_{j-1},d_{j}}$ is continuous. According to Lemma \ref{lem:stab_gaussienne} and the concatenation formula of Lemma \ref{lem:compo_gaussien}, the function $\mathcal{C} : (P_1, \dots, P_{d-1}) \in \GG(\mu_1,\mu_2) \times \dots \times \GG(\mu_{d-1},\mu_d) \mapsto P_1 \circ \dots \circ P_{d-1}$ is continuous. The continuity of the composition follows directly from the continuity of projections.
\end{proof}

We now recall a theorem of Kellerer,  primary tool of the proof (see \cite[Theorem $1$]{kellerer_markov-komposition_1972}). This theorem is an existence result of a Markov measure satisfying certain constraints. It generalizes  the standard Kolmogorov extension theorem of a Markov process fitting a consistent family of two-dimensional laws (just take $\NN^{s,t} = \{\mu_{s,t}\}$ below).
\begin{them}\label{Them:Kellerer_consistency}
	Let $T$ be an interval, and $ (\mu_t)_{t \in {T}} $ a family of probability measures on some Polish space \( E \). For all \( s < t \in T^2\), let $ \mathcal{N}_{s,t} $ be a subset of $ \mathcal{P}(E^2)$.  Assume that, for all $s<t \in T^2$:
	\begin{longlist}
		\item[(1)] $  \mathcal{N}_{s,t} $ is non-empty;
		\item[(2)] \( \mathcal{N}_{s,t} \subset \Marg(\mu_s, \mu_t) \);
		\item[(3)] \( \mathcal{N}_{s,t} \) is closed with respect to the weak topology;
		\item[(4)] For all \( r < s < t \in T^3\) and  \( (P, P') \in \mathcal{N}_{r,s} \times \mathcal{N}_{s,t} \), \( P \cdot P' \in \mathcal{N}_{r,t} \);
		\item[(5)] For every \( d \geq 1 \) and \( t_1 < \ldots < t_d \in T^d \), if \( (Q^{n}_{t_i, t_{i+1}})_{n \geq 1} \in \left(\mathcal{N}_{t_i, t_{i+1}}\right)^{\N^*} \) is a sequence that converges weakly to a measure \( Q_{t_i, t_{i+1}} \) for all $i =1, \dots, d-1$, then the sequence \( \left(Q^{n}_{t_1, t_2} \circ \ldots \circ Q^{n}_{t_{d-1}, t_d}\right)_{n \geq 1} \)  tends weakly to \( Q_{t_1, t_2} \circ \ldots \circ Q_{t_{d-1}, t_d} \).
	\end{longlist}
	Then, there exists a Markov measure \( P \in \Marg((\mu_t)_{t \in T}) \) such that, for all \( s < t \in T^2\), \( \proj^{s,t}_{\#}P \in \mathcal{N}_{s,t} \).
\end{them}
If $R =(r_1, \dots, r_p)$ and $S = (s_1, \dots, s_q)$ are such that $r_p = s_1$, define $R+S := (r_1, \dots,r_p,s_2,\dots,s_p) \in \SS_{[r_1,s_q]}.$ 

\begin{proof}[Proof of Proposition \ref{proA:existence}]
	For all $s<t \in T^2$ and $\s >0$, define $$ \NN_{s,t}^{\s} = \ens{P^{s,t}_{\{R\}}}{R \in \SS_{[s,t]} \text{ and } \s_R \leq \s}$$ and $\NN_{s,t} = \cap_{\s >0} \hspace{0.05cm} \overline{\NN_{s,t}^\s}.$  In order to apply Theorem \ref{Them:Kellerer_consistency}, we establish that the conditions $(1)$ to $(5)$ are fulfilled. Since $\Marg(\mu_s,\mu_t)$ is compact and $\GG_{d+d}$ is closed, $\GG(\mu_s,\mu_t) = \Marg(\mu_s,\mu_t)\cap \GG_{d+d}$ is also compact. Moreover, the composition formula of Lemma \ref{lem:compo_gaussien} implies $ \NN_{s,t}^\s \subset \GG(\mu_s,\mu_t)$ and $\NN_{s,t}^\s$ is obviously not empty. Finally $\NN_{s,t}$ is a decreasing intersection of non-empty compact subsets of $\GG(\mu_s,\mu_t)$, thus is a non-empty compact subset of $\GG(\mu_s,\mu_t)$. This shows that conditions $(1), (2)$ and $(3)$ hold. For condition $(4)$, consider $Q_1 \in \NN_{s,t}$, $Q_2 \in \NN_{t,u}$ and $\s > 0.$ There exists two sequences of partitions $(R_n)_{n \geq 1} \in {(\SS_{[s,t]})}^{\N^*}$ and $(S_n)_{n \geq 1} \in {(\SS_{[t,u]})}^{\N^*}$ such that $Q_1 = \lim_{n \to +\infty} P_{\{R_n\}}^{s,t} $, $Q_2 = \lim_{n \to +\infty} P_{\{S_n\}}^{t,u}$ and $\max(\s_{R_n},\s_{S_n}) \leq \s.$ According to Lemma \ref{lem:compo_gaussien_stab}, we get $P_{\{R_n + S_n\}}^{s,u} = P_{\{R_n\}}^{s,t} \cdot P_{\{S_n\}}^{t,u}  \dcv{n \to + \infty} Q_1 \cdot Q_2.$ Since $\s_{R_n+ S_n} = \max(\s_{R_n},\s_{S_n}) \leq \s$, it follows that $Q_1 \cdot Q_2 \in \overline{\NN_{s,u}^\s}$. Since this holds for all $\s > 0$, $Q_1 \cdot Q_2 \in \NN_{s,u},$  and condition $(4)$ is satisfied. For $(5)$, just recall that $\NN_{s,t} \subset \GG(\mu_s,\mu_t)$ and apply Lemma \ref{lem:compo_gaussien_stab}. Therefore Theorem \ref{Them:Kellerer_consistency} applies, which ensures the existence of a Markov measure $ P' \in \PP(\R^T)$ such that, for all $s<t \in T^2$, $P'^{s,t} \in \NN_{s,t}$. This means precisely that $P'$ is a weak local Markov transform of $P$.
\end{proof}

Proposition \ref{proA:existence} extends the result of Theorem \ref{them:fonda_NB} for Gaussian measures. First, our result is  valid for any measures $\mu_t \in \GG_d^*, \ t \in T$  with $d \geq 1$, whereas Theorem \ref{them:fonda_NB} only applies when $d=1$. Additionally, if $d = 1$, we do not require that $P \in \Marg((\mu_t)_{t \in T})$ has increasing kernels. For a Gaussian process, having increasing kernels is equivalent to having non-negative covariance function. In detail, for $P \in \Marg((\mu_t)_{t \in T})$  with covariance function $K$, by the first point of Lemma \ref{lem:gauss_cond}, $P^{s,t}(dx,dy) = \mu_t(dx)k^{s,t}_x(dy)$ with $k^{s,t}_x = \NN\left( K(s,t)/K(s,s)x, K(t,t)-K(s,t)^2/K(s,s) \right)$. Recall that, for all $(x, y , \s) \in \R^2 \times \R_+^*$, $\NN(x, \s) \leq_S \NN(y,\s)$ if and only if $x \leq y$. Thus, $P^{s,t}$ has increasing kernels if and only if $K(s,t) \geq 0$, proving that $P$ has increasing kernels precisely when $K$ takes only non-negative values.

\section{Proof of Theorem \ref{them:counterexample_intro}}\label{sec:counterexample} In order to construct a (stationary) Gaussian measure which has several weak local Markov transforms, we start by proving an analogue of Lemma \ref{them:Markov_stationnaire} for local Markov transforms of Gaussian processes.

\begin{lem}\label{pro:finite_dim_conv_stati}
	Consider a continuous stationary positive semi-definite kernel $K :(s,t) \in T^2 \mapsto \ti{K}(t-s) \in \R$  with $\ti{K}(0) = 1.$ Let $P$ be a centered Gaussian measure with covariance function $K$ and set $L^K : h \mapsto h^{-1}(1-\ti{K}(h)).$ Assume $\a \in [0,+\infty]$ is a cluster point of $L^K$ at $0^+$ and let $(s_n)_{n \geq 1}$ be a positive sequence converging to zero such that $\lim_{n \to + \infty} L^K(s_n) = \a.$ Then, defining $R_n := ( s_n  \Z ) \cap [-n,n]$, we have $P_{[R_n]}  \xrightarrow[n \to + \infty]{\fd} P_\a$. In particular, the measure $P_\a$ is a weak global Markov transform of $P.$ 
\end{lem}
\begin{proof}
	Let $(s_n)_{n \geq 1}$ and $(R_n)_{n \geq 1}$ be as in the statement. By Remark \ref{remq:finite_dim_conv_gaussian}, we have to show that, for all $s<t \in \R^2$, $\lim_{n \to +\infty} P_{[R_n]}^{s,t} = P_\a^{s,t}$ . Note that, for all $n \geq \max(-s,t)$, we have $$P_{[R_n]}^{s,t} = P_{\{\{s,t\}\cup(R_n \cap ]s,t[)\}}^{s,t} = P_{\{R_n^{s,t}\}}^{s,t},$$ where $R_n^{s,t} := \left((s_n \Z) \cap ]s,t[\right) \cup \{s,t\}$. Writing $R_n^{s,t} := (t_0^n, \dots, t_{m_n}^n)$. For all $k \in \{1, \dots, m_n-2\}$, we have 
	\begin{align*}
		\ln(\ti{K}(s_n)) &= \ln(1 + [\ti{K}(s_n)-1]) = -s_nL(s_n)\left[1+ \ee\left(\ti{K}(s_n)-1\right)\right]\\
		&= -(t_{k+1}^n - t_k^n)L(s_n)\left[1+ \ee\left(\ti{K}(s_n)-1\right)\right],
	\end{align*}
	where $\ee$ is defined in Notation \ref{nott:DL_Taylor}. Hence, $ u_n := \sum_{k=0}^{m_n-1} \ln\left(K(t_{k}^n,t_{k+1}^n)\right) = \sum_{k=0}^{m_n-1} \ln\left(\ti{K}(t_{k+1}^n-t_k^n)\right)$ satisfies
	\begin{align*}
		u_n &= \ln(\ti{K}(t_1^n-s)) + \ln(\ti{K}(t-t_{m_n-1}^n)) + \sum_{k=1}^{m_n-2} \ln(\ti{K}(s_n)) \\
		&= \ln(\ti{K}(t_1^n-s)) + \ln(\ti{K}(t-t_{m_n-1}^n)) - \sum_{k=1}^{m_n-2}(t_{k+1}^n - t_k^n)L(s_n)\left[1+ \ee\left(\ti{K}(s_n)-1\right)\right] \\
		&= \ln(\ti{K}(t_1^n-s)) + \ln(\ti{K}(t-t_{m_n-1}^n)) - (t_{m_n-1}^n-t_1^n) L(s_n)\left[1+ \ee\left(\ti{K}(s_n)-1\right)\right].
	\end{align*}
	Given that $\ti{K}(0) = 1$ and $K$ is continuous, it follows that $ \lim_{n \to +\infty} \sum_{k=0}^{m_n-1} \ln\left(K(t_{k}^n,t_{k+1}^n)\right) = 0 + 0 - \a (t-s) = -\a(t-s)$. By Proposition \ref{pro:formulation_calculatoire}, this concludes.
\end{proof}

The guiding result for our construction is a theorem of Bochner which characterizes the continuous positive semi-definite stationary kernels. For a proof, we refer to \cite[Paragraph $8$]{bochner_monotone_1933}  (or \cite[Page $208$]{gihman_theory_1974} for a more recent presentation). Let $\hat{\mu} : t \in \R \mapsto \int_{\R} \exp(itx) \mu(dx)$ denote the Fourier transform of a positive finite measure $\mu \in \MM_+(\R).$
\begin{them}[Bochner\footnote{In general, Theorem \ref{them:Bochner} is stated for $\mathbb{C}$-valued kernels, but it is straightforward that this implies our reformulation with $\R$-valued kernels.}]\label{them:Bochner}
	Let $K : \R^2 \to \R$ be a symmetric map. The following conditions are equivalent: 
	\begin{longlist}
		\item There exists $\mu \in \MM_+(\R)$ such that, for every $(s,t) \in \R^2$, $K(s,t) = \hat{\mu}(t-s)$. 
		\item The function $K$ is a continuous positive semi-definite stationary kernel.
	\end{longlist}
	The measure $\mu$ is then unique and is called the spectral measure associated with $K$.
\end{them}	
In the light of Lemma \ref{pro:finite_dim_conv_stati} and Theorem \ref{them:Bochner}, in order to obtain a stationary Gaussian measure $P$ such that $P_\alpha$ is a weak global Markov transform of $P$ for all $\alpha \in [0,+\infty]$, it is sufficient to find a symmetric probability measure $\mu \in \PP(\R)$ such that any $\alpha \in [0,+\infty]$ is a cluster point of the decay rate of $\hat{\mu}$ at $0^+$. Indeed, note that if $\mu$ is a real-valued symmetric probability measure, then the function $K : (s,t) \in \R^2 \mapsto \hat{\mu}(t-s)$  is real valued and symmetric. According to Theorem \ref{them:Bochner}, this implies that $K$ is a stationary positive semi-definite kernel, with constant variance equal to $\hat{\mu}(0) = \mu(\R) = 1$.  Thus, for every cluster point $\a \in [0,+\infty]$ of the decay rate of $\hat{\mu}$ at point $0^+$, according to Lemma \ref{pro:finite_dim_conv_stati}, the centered Gaussian process $P$ with covariance function $K$ admits $P_{\a}$ as weak local Markov transform.

Denoting by $\emph{\sgn} : \R \to \{-1,0,1\}$ the usual sign function, our strategy to construct such $\mu$ is to consider a symmetric probability measure of the form $\g := \sum_{|k| \geq 2 } a_{|k|} \d_{\sgn(k)b_{|k|}}$ whose Fourier transform has an infinite decay rate at $0^+$. Then, to obtain our measure $\mu$, we  ``mix'' $\g$ with $\d_0$, whose Fourier transform has a decay rate converging to $0$ at point $0^+$. More precisely, we will recursively construct a set $S \subset \left(\N \cap [2,+\infty[ \right)$ and put \[y_k := \begin{cases}
	b_k &\text{if } k \in S \\ 0 &\text{otherwise}	
\end{cases},\] in order to define our measure $\mu$ by \[\mu := \sum_{|k| \geq 2 } a_{|k|} \d_{\sgn(k)y_{|k|}} = \sum_{|k| \in S} a_{|k|}\d_{\sgn(k)b_{|k|}} + \left( \sum_{|k| \notin S} a_{|k|} \right) \d_0.\] Since 
\begin{equation}\label{eq:trans_fourier_mu_gamma}
	\begin{cases}\hat{\g}(t) = \sum_{k \geq 2} a_k(e^{itb_k} + e^{-it b_k}) = 2 \sum_{ k \geq 2} a_k \cos(tb_k)\\ \hat{\mu}(t) = \sum_{k \geq 2} a_k(e^{ity_k} + e^{-it y_k}) = 2 \sum_{ k \geq 2} a_k \cos(ty_k) \end{cases},
\end{equation}
the decay rates of $\hat{\g}$ and $\hat{\mu}$ at $0$ are given by \begin{equation}\label{eq:calcul_slope}
	\begin{cases}
		t^{-1}(1-\hat{\g}(t)) = 2 t^{-1}\sum_{k \geq 2} a_k \left( 1- \cos(t b_k)  \right)\\ t^{-1}(1-\hat{\mu}(t)) = 2 t^{-1}\sum_{k \in S} a_k \left( 1- \cos(t b_k)  \right)
	\end{cases}.
\end{equation} 
Set $a := 1/2$, $b := 3$, and for all $k \geq 2$, $a_k := a^k$, $b_k := b^k$. We have $ab>1$ and \eqref{eq:trans_fourier_mu_gamma} shows that $\hat{\g}/2$ is the well-known continuous nowhere differentiable Weierstrass function \cite{weierstras_uber_1872}. In Point \eqref{point:cv_derivee_Fourier} of Lemma \ref{lem:va_fourier_preli} below, we rely on an article by Hardy \cite{hardy_weierstrasss_1916} to prove that  $\lim_{t \to 0^+} t^{-1}(1-\hat{\g}(t)) = + \infty$. We now need to find a recursive construction of $S$ such that the lacunary series $t^{-1}(1-\hat{\mu}(t))$ of $t^{-1}(1-\hat{\g}(t))$  admits the elements of $[0,+\infty]$ as cluster points.  Points \eqref{point:cv_derivee_Fourier_queue}-\eqref{point:cv_derivee_Fourier_modif} of Lemma \ref{lem:va_fourier_preli} are useful to define the sequence $(y_k)_{k \geq 1}$ and to show that the resulting measure $\mu$ has the wanted property. The construction itself is done in Proposition \ref{pro:va_fourier}, using the tools from Lemma \ref{lem:va_fourier_preli}.
\begin{lem}\label{lem:va_fourier_preli}
	Set $a := 1/2$, $b := 3$, and for all $k \geq 2$, $a_k := a^k$, $b_k := b^k$.
	\begin{longlist}
		\item Thus, we have $\lim_{x \to + \infty}	x \sum_{k \geq 2} a_k \left(1 - \cos \left( \frac{b_k}{x} \right) \right) =  +\infty.$
		\item\label{point:cv_derivee_Fourier_queue} For all $(x_k)_{k \geq 1} \in \R^{\N^*}$, we have $\lim_{x \to +\infty} x \sum_{k \geq \lfloor x \rfloor +1} a_k \left(1 - \cos \left( \frac{x_k}{x} \right) \right) = 0.$
		\item\label{point:cv_derivee_Fourier_milieu} For all $p \geq 1$ and $((n_i,m_i))_{i =1,\dots,p} \in {\left(\N^2\right)}^p,$ the map \[g_{(n_1,m_1),\dots,(n_p,m_p)} : x \in \R_+^*  \mapsto x \sum_{l=1}^p \sum_{k=n_l}^{m_l} a_k \left( 1-\cos\left(\frac{b_k}{x}\right) \right) \in \R_+\] satisfies  $\lim_{x \to + \infty} g_{(n_1,m_1),\dots,(n_p,m_p)}(x) = 0.$
		\item\label{point:cv_derivee_Fourier_finale} For  $n \geq 2$, the map  \[f_n : x \in \R \mapsto x\sum_{k = n}^{\lfloor x \rfloor } a_k \left( 1 - \cos \left( \frac{b_k}{x} \right) \right) \in \R_+\] satisfies $ \lim_{x \to + \infty} f_n(x) =  +\infty.$
		\item\label{point:cv_derivee_Fourier_modif} Consider an increasing sequence of integers $(n_k)_{k \geq 1}$ such that $n_0 = 2$ and define the sequence $(y_k)_{k \geq 0}$ by \[ y_k = \begin{cases} b_k & \text{if } k \in \bigcup_{i \geq 0} \{n_{2i},\dots, n_{2i+1}-1\} \\ 0 &\text{if } k \in \bigcup_{i \geq 0} \{n_{2i+1},\dots, n_{2(i+1)-1} \} \end{cases}.\] Then the map \[f : x  \in \R_+^* \mapsto x\sum_{k=2}^{ \lfloor x \rfloor } a_k \left( 1 - \cos \left( \frac{y_k}{x} \right) \right)\] satisfies \[f(x) =  \begin{cases} g_{(n_0,n_1-1),\dots,(n_{2(i-1)},n_{2i-1}-1)}(x) + f_{n_{2i}}(x) & \text{if }  x \in [n_{2i},n_{2i+1}[ \\ g_{(n_0,n_1-1),\dots,(n_{2i},n_{2i+1}-1)}(x) &\text{if } x \in [n_{2i+1},n_{2(i+1)}[ \end{cases}.\]
	\end{longlist}
\end{lem}
\begin{proof}
	\begin{longlist}
		\item\label{point:cv_derivee_Fourier} According to  \cite[Points $2.41$-$2.42$]{hardy_weierstrasss_1916}, we have $\lim_{h \to 0^+} h^{-1} \sum_{k \geq 0} a^k (1- \cos(b^k \pi h)) = + \infty$. Since $\lim_{h \to 0^+} h^{-1} a^0 (1-\cos(b^0\pi h) = \lim_{h \to 0^+} h^{-1} a^1 (1-\cos(b^1\pi h) = 0$, the change of variable $x^{-1} = \pi h$ gives the wanted result.
		\item For  $x \geq 2$, 
		$ {\small 0 \leq x \sum_{k \geq \lfloor x \rfloor +1} a_k \left( 1 - \cos\left( \frac{x_k}{x} \right) \right) 
			\leq 2x \sum_{k \geq \lfloor x \rfloor +1} a^k 
			= 2x\frac{a^{\lfloor x \rfloor + 1}}{1-a}} ,$
		which proves the result.
		\item For $x \in \R^*_+$, $$ 0 \leq x \sum_{k=1}^p \sum_{l=n_k}^{m_k} a_l \left( 1-\cos\left(\frac{b_l}{x}\right) \right) \leq x \sum_{k=1}^p \sum_{l=n_k}^{m_k} a_l \frac{b_l^2/x^2}{2} = \frac{1}{x} \sum_{k=1}^p \sum_{l=n_k}^{m_k} \frac{a_l b_l^2}{2},$$ which proves the result.
		\item According to Points \eqref{point:cv_derivee_Fourier}, \eqref{point:cv_derivee_Fourier_queue}, and \eqref{point:cv_derivee_Fourier_milieu}, the equality \[f_n(x)  = x\sum_{k \geq 2}^{\infty} a_k \left(1 - \cos \left( \frac{b_k}{x} \right) \right) - x\sum_{k=2}^{n-1} a_k \left(1 - \cos \left( \frac{b_k}{x} \right) \right) - x \sum_{k \geq \lfloor x \rfloor + 1} a_k \left(1 - \cos \left( \frac{b_k}{x} \right) \right)  \] implies $\lim_{x \to + \infty} f_n(x) = + \infty -0-0 = + \infty.$
		\item The computation is straightforward.
	\end{longlist}
\end{proof}

\begin{pro}\label{pro:va_fourier}
	There exists a sequence $(y_k)_{k \geq 1} \in \R_+^{\N^*}$ such that the cluster points of the decay rate at point $0^+$ of the Fourier transform of $\mu := \sum_{|k| \geq 2 } 2^{-|k|} \d_{\sgn(k)y_{|k|}}$ are the elements of $[0,+\infty]$.
\end{pro}
\begin{proof}
	We use the functions $f_n$ and $g_{(n_1,m_1), \dots , (n_p,m_p)}$, as defined of Lemma \ref{lem:va_fourier_preli}. According to Points \eqref{point:cv_derivee_Fourier_milieu} and \eqref{point:cv_derivee_Fourier_finale} of  Lemma \ref{lem:va_fourier_preli}, we can define a sequence $(n_i)_{i \geq 2}$ by \[\accol{n_0 := 2 \\ \forall i \geq 0, n_{2i+1} := \inf \ens{n > n_{2i} }{ f_{n_{2i}}(n-1) > i }\\ \forall i \geq 0, n_{2(i+1)} := \inf \ens{n > n_{2i+1} }{g_{(n_0,n_1-1),\dots,(n_{2i},n_{2i+1}-1)}(n-1) < 1/i}}.\] As in Point \eqref{point:cv_derivee_Fourier_modif} of Lemma \ref{lem:va_fourier_preli}, we associate a sequence $(y_k)_{k \geq 2}$ and a function $f$ with our sequence $(n_i)_{i \geq 0}$. Since for all $i \in \N$, $n_{2i+1} - 1 \in [n_{2i},n_{2i+1}[$ and $n_{2(i+1)} -1 \in [n_{2i+1}, n_{2(i+1)}[$, still by Point \eqref{point:cv_derivee_Fourier_modif} of Lemma \ref{lem:va_fourier_preli} 
	$$f(n_{2i+1}-1) = g_{(n_0,n_1-1),\dots ,(n_{2(i-1)},n_{2i-1}-1)}(n_{2i+1}-1) + f_{n_{2i}}(n_{2i+1}-1)\geq f_{n_{2i}}(n_{2i+1}-1) > i$$
	and $$f(n_{2(i+1)}-1) = g_{(n_0,n_1-1),\dots,(n_{2i},n_{2i+1}-1)}(n_{2(i+1)}-1) < 1/i$$ for $i \geq 1$.
	Set $\mu := \sum_{|k| \geq 2 } 2^{-|k|} \d_{\sgn(k)y_{|k|}}$ and let $L :t \in \R_+^* \to t^{-1}(1-\hat{\mu}(t))$ denote the decay rate of $\hat{\mu}$ at $0^+$. As in Equation \eqref{eq:calcul_slope},  \[L(1/x) = 2x\sum_{k \geq 2} a_k \left( 1-\cos(y_k/x)  \right) = 2(f(x) +  g(x)),\] where $g(x) := x\sum_{k \geq \lfloor x \rfloor + 1} a_k \left( 1- \cos(y_k/x)  \right) $. According to Point \eqref{point:cv_derivee_Fourier_queue} of Lemma \ref{lem:va_fourier_preli}, $\lim_{x \to +\infty} g(x) = 0.$ Thus, writing $s_i :=  \frac{1}{n_{2i+1}-1}$ and  $t_i :=  \frac{1}{n_{2(i+1)}-1}$, we get $ L(s_i) =2 f(n_{2i+1}-1) + 2g(n_{2i+1}-1) \dcv{ i \to + \infty} + \infty + 0 = +\infty$ and $ L(t_i) =  f(n_{2(i+1)}-1) + g(n_{2(i+1)}-1) \dcv{i \to + \infty} 0$, so that $0 $ and $+\infty$ are cluster points of $L$ at $0^+.$ Since $L$ is continuous, by the intermediate value theorem, any element of $[0,+\infty]$ is a cluster points of $L$. Since $L$ is non-negative, all its cluster points are elements of $[0,+\infty]$, which concludes.
\end{proof}

Before proving Theorem \ref{them:counterexample_intro}, we show that a weak local Markov transform of a Gaussian measure is a Gaussian measure. 

\begin{pro}\label{pro:Gaussianite_limite}
	Fix an interval $T \subset \R$, $d \geq 1$, $(\mu_t)_{t \in T} \in \left( \GG_d^* \right)^T$ and let $P \in \PP((\R^d)^T)$ be a Gaussian measure. If $P' \in \PP((\R^d)^T)$ is a weak local Markov transform of $P$, then $P'$ is also a Gaussian measure.
\end{pro}
\begin{proof}
	By Remark \ref{remq:centered_enough}, we can assume all the involved measures are centered. Now, fix $s<t \in T^2$ and $(R_n)_{n \geq 1} \in {(\SS_{[s,t]})}^{\N^*}$ such that $\lim_{n \to + \infty} P_{\{R_n\}}^{s,t} = P'^{s,t}$. By Lemma \ref{lem:compo_gaussien} $P_{\{R_n\}}^{s,t} \in \GG(\mu_s,\mu_t)$, and by Lemma \ref{lem:stab_gaussienne} a limit of a sequence of centered Gaussian measures is centered Gaussian. This implies $P'^{s,t}$ is Gaussian. Since $P'$ is a Markov measure, for all $t_1 < \dots < t_p \in T^p$, $P'^{t_1,\dots,t_p} = P'^{t_1,t_2} \circ \cdots \circ P'^{t_{p-1},t_p}.$ By Lemma \ref{lem:compo_gaussien}, this implies $P'^{t_1,\dots,t_p}$ is Gaussian and proves the statement.
\end{proof}

We can now prove Theorem \ref{them:counterexample_intro}.
\begin{proof}[Proof of Theorem \ref{them:counterexample_intro}]
	Expressed in terms of measures, we want to construct a stationary Gaussian measure $P \in \PP\left(\R^\R\right)$ such that for all $\alpha \geq 0 $, $P_\alpha$ is a weak global Markov transform and whose set of weak local Markov transforms is the set of all the measures $P'$ satisfying:
	\begin{longlist}
		\item\label{Point:is_Gaussian} The measure $P'$ is centered Gaussian with constant variance function equal to $1$;
		\item\label{Point:is_non_neg} The covariance function of $P'$ is non-negative;
		\item\label{Point:is_Markov} The measure $P'$ is Markov.
	\end{longlist}
	\textbf{Construction of $P$:} According to Proposition \ref{pro:va_fourier}, there exists a symmetric probability measure $\mu$ such that the cluster points of the decay rate of $\hat{\mu}$ at $0^+$ are the elements of $[0,+\infty]$. Since $\mu$ is symmetric, the function $K : (s,t) \in \R^2 \mapsto \hat{\mu}(t-s)$  is real-valued and symmetric. Hence by Theorem \ref{them:Bochner}, $K$ is a stationary positive semi-definite kernel, with constant variance equal to $\hat{\mu}(0) = \mu(\R) = 1$. Let $P$ denote the centered Gaussian measure with covariance function $K.$ \\
	\textbf{Weak global Markov transform:} By construction of $P$, all $\alpha \geq 0$ is a cluster point of the decay rate of $\hat{\mu}$ at $0^+$, that is a cluster point of the decorrelation rate of $K$ at $0^+$. By Lemma \ref{pro:finite_dim_conv_stati}, $P_\alpha$ is a weak global Markov transform of $P$.\\
	\textbf{Weak local Markov transform:}	
	Assume $P'$ is a measure satisfying conditions \eqref{Point:is_Gaussian}, \eqref{Point:is_non_neg} and \eqref{Point:is_Markov}. To show that $P'$ is a weak local  Markov transform of $P$, fix $s<t \in \R^2$. Since  $0 \leq K'(s,t) \leq K'(s,s)^{1/2}K'(t,t)^{1/2} = 1$,  $\a := -|t-s|^{-1} \ln(K'(s,t))$ is well defined and  $\a \in [0,+\infty]$. Moreover $K'(s,t) = \exp(-\a |t-s|) = K_\a(s,t)$.  Since $P_\a$ is a weak global Markov transform of $P$ it is also a weak local Markov transform. Hence, there exists partitions $R_n = (t_k^n)_{k ={1},\dots,{m_n}} \in \SS_{[s,t]}, n \geq 1$ of $[s,t]$ such that $\lim_{n \to +\infty} \s_{R_n} = 0$ and $\lim_{n \to + \infty} K(s,t_1^n) \dots K(t_{m_n-1}^n,t) = K_\a (s,t) = K'(s,t).$ Thus by Proposition \ref{pro:formulation_calculatoire}, $P'$ is a weak local Markov transform of $P.$ Conversely, suppose  $P'$ is a weak local Markov transform of $P$ with covariance function $K'$. We aim to show that $P'$ satisfies conditions \eqref{Point:is_Gaussian}, \eqref{Point:is_non_neg} and \eqref{Point:is_Markov}. Condition \eqref{Point:is_Markov} holds by definition of a weak local Markov transform. By Proposition \ref{pro:Gaussianite_limite}, $P'$ is Gaussian. Moreover $P'$ has the same $1$-marginals as $P$, thus condition \eqref{Point:is_Gaussian} holds. Since $P'$ is a weak local Markov transform of $P$, there exists partitions $R_n = (t_k^n)_{k ={1},\dots,{m_n}} \in \SS_{[s,t]}, n \geq 1$ of $[s,t]$ such that  $\lim_{n \to +\infty} K(t_0^n,t_1^n) \cdots K(t_{m_n-1}^n,t_{m_n}^n) = K'(s,t).$ As $\hat \mu$ is continuous and $\hat \mu (0) = 1 > 0$, this proves Condition \eqref{Point:is_non_neg}. 
\end{proof}
In particular, for a measure $P$ as in the previous proof, for any measurable function $\a : \R \to [0,+\infty]$, $P_\a$ is a weak local Markov transform of $P.$ As one can see by taking $\a : t \mapsto 0 \cdot 1_{t \leq 1} +  (+ \infty) \cdot 1_{t>1}$, this also proves that a weak Markov transform of a stationary measure is not always a stationary measure. 

\appendix

\section{Technical result on Gaussian measures}

We gather in this appendix some results about Gaussian measures. We start by stating a standard lemma that characterizes the weak convergence of (centred) Gaussian measures through the convergence of their covariance matrices. We refer for instance to \cite[Ch. 8, Theorem 3]{bergstrom_weak_1982} for a proof.
\begin{lem}\label{lem:stab_gaussienne}
	Consider centred Gaussian measures $\mu_n \in \GG_d, n \geq 1$ and $\mu_0 \in \PP(\R^d)$. The following conditions are equivalent.
	\begin{longlist}
		\item  The measure $\mu_0$ is centered Gaussian and $\Sigma_{\mu_0} = \lim_{n \to + \infty} \Sigma_{\mu_n} .$
		\item  The sequence $(\mu_n)_{n \geq 1}$ weakly converges to $\mu_0$.
	\end{longlist}		 
\end{lem}
The following result regarding the conditional laws of Gaussian vectors is standard and we direct the reader toward \cite[Chapter $8$, Section $9$]{von_mises_mathematical_1964} for a proof.
\begin{lem}[Conditioning of Gaussian measures]\label{lem:gauss_cond}
	Fix $(m,n) \in {\N^*} \times \N^*$. 
	\begin{longlist}
		\item Consider a Gaussian measure  $\pi \in\PP(\R^{m+n})$ and define $\mu := {\proj^{1, \cdots,m}}_\# \pi$, $\nu := {\proj^{m+1, \cdots,m+n}}_\# \pi$, $\Sigma_{\pi} := \Sigma_{\pi}^{m,n} \in \MM_{m,n}(\R)$. If $\mu \in \GG_m^*$ and $k : \R^m \to \PP(\R^n)$ is a probability kernel such that $\pi(dx,dy) = \mu(dx)k_x(dy)$, then \[\mu(dx)-\textnormal{a.s.} , \ k_x = \NN\left(\Sigma_{\pi}^t \Sigma_\mu^{-1}x, \Sigma_\nu- \Sigma_{\pi}^t \Sigma_\mu^{-1}\Sigma_{\pi}\right).\]
		\item Consider $\mu \in \GG_m^*$, $(A,\Gamma) \in \MM_{n,m}(\R) \times \MM_n(\R)$ and $k : x \in \R^m \to \NN(Ax,\Gamma) \in \PP(\R^n)$. Then, writing $\pi(dx,dy) = \mu(dx)k_x(dy)$, it follows that \[ \pi = \NN\left( 0_{\R^{m+n}},\begin{pmatrix}
			\Sigma_\mu & \Sigma_\mu A^\top \\ A \Sigma_\mu & \Gamma + A \Sigma_\mu A^\top
		\end{pmatrix}\right) \in \GG_{m+n}.\]
	\end{longlist}
\end{lem}
Using Lemma \ref{lem:gauss_cond}, we derive an explicit formula for the concatenation and the composition of Gaussian measures.
\begin{lem}\label{lem:compo_gaussien}
	Consider $(\mu_1, \dots, \mu_p) \in \GG_{d_1}^* \times \cdots \times \GG_{d_p}^*$ and $(P_i)_{i ={1},\dots,{p-1}} \in \prod_{i=1}^{p-1} \GG(\mu_i,\mu_{i+1})$. Define  $\Sigma_{i,i+1} := \Sigma_{P_i}^{d_i,d_{i+1}}$ and $\Sigma_{i,i} := \Sigma_{\mu_i}$.
	\begin{description}
		\leftskip=0.3in
		\item[Concatenation formula:] $P_1 \circ \cdots \circ P_{p-1} = \NN \left( 0, \left( A_{i,j} \right)_{1 \leq i,j \leq p} \right),$ where  $A_{i,j} \in \MM_{d_i,d_j}(\R)$ is defined by
		\[A_{i,j} =
		\begin{cases}
			\Sigma_{i,i+1} \Sigma_{i+1,i+1}^{-1} \Sigma_{i+1,i+2} \cdots \Sigma_{j-1,j-1}^{-1} \Sigma_{j-1,j}\in \MM_{d_i,d_{j}}(\R) & \text{ if } i < j \\
			\Sigma_{i,i} & \text{ if } i=j \\
			A_{j,i}^\top &\text{ if } j < i \\
		\end{cases}
		\]
		\item[Composition formula:] $ P_1 \cdot \ \dots \ \cdot  P_{p-1} = \NN \left( 0 , \Sigma\right)$, with $$\Sigma = \begin{pmatrix}
			\Sigma_{\mu_1} & \Sigma_{1,2} \Sigma_{2,2}^{-1} \Sigma_{2,3} \dots \Sigma_{p-1,p-1}^{-1} \Sigma_{p-1,p} \\ (\Sigma_{1,2} \Sigma_{2,2}^{-1} \Sigma_{2,3} \dots \Sigma_{p-1,p-1}^{-1} \Sigma_{p-1,p})^\top & \Sigma_{\mu_p}
		\end{pmatrix}. $$
	\end{description}
\end{lem}
\begin{proof}
	The proof of the concatenation formula is based on a recursion on $p$, Lemma \ref{lem:gauss_cond}, and the decomposition  $P_{1,2} \circ P_{2,3} = \mu_2 (dx_2) [k^{2,1}(x_2, \cdot) \otimes k^{2,3}(x_2, \cdot)](dx_1,dx_3)$ (see Notation/Definition \ref{def:compo}). The composition formula then follows directly from the concatenation formula and the projection onto $\R^{d_1}\times \R^{d_p}$.
\end{proof}

When the covariance matrices  are identity matrices, for all $i<j \in \{1,\dots,d\}^2$, $A_{i,j} = \Sigma_{i,i+1} \dots \Sigma_{j-1,j} \in \MM_{d_i,d_j}$ is the product of the $j-i$ successive matrices  $\Sigma_{k,k+1}$, beginning at $k =i$. Lemma \ref{lem:compo_gaussien} allows us to recover\footnote{This criterion seems to be known for a long time \cite{borisov_criterion_1982}, but to the best of our knowledge, not its multi-dimensional version.} a  characterization of the Markov property for Gaussian measures.
\begin{pro}\label{pro:criteria_gauss_Markov}
	Fix $d \geq 1$, $T \subset \R$, and $\mu_t \in \GG_d^*, t \in T$. Let $P \in \Marg((\mu_t)_{t \in T})$ denote a centred Gaussian measure. For all $s<t \in T^2$, define $\Sigma_{t,t} := \Sigma_{\mu_t}$ and $\Sigma_{s,t} := \Sigma_{P^{s,t}}^{d,d}$.  The measure $P$ is Markov if, and only if, 
	\begin{equation}\label{eq:comp_dim}
		\forall s<t<u \in T^3, \Sigma_{s,u} = \Sigma_{s,t} \Sigma_{t,t}^{-1} \Sigma_{t,u}.
	\end{equation}
\end{pro}

\begin{proof}
	If $P$ is Markov, then for all $s<t<u \in T^3$, according to the composition formula, it follows that $\Sigma_{s,u} = \Sigma_{P^{s,u}}^{d,d} = \Sigma_{P^{s,t} \cdot P^{t,u}}^{d,d} = \Sigma_{P^{s,t}}^{d,d} \Sigma_{\mu_t}^{-1} \Sigma_{P^{t,u}}^{d,d} = \Sigma_{s,t}\Sigma_{t,t}^{-1} \Sigma_{t,u}.$ For the converse implication, assume Condition \eqref{eq:comp_dim} holds, and we aim to show that, for all $t_0 < \dots < t_m \in T^{m+1}$, $P^{t_0,\dots,t_m} = P^{t_0,t_1} \circ \dots \circ P^{t_{m-1},t_m}$. For any  $i<j \{1,\dots, d-1\}^2$, applying Condition \eqref{eq:comp_dim} recursively, we obtain $\Sigma_{t_i,t_{i+1}} \Sigma_{t_{i+1},t_{i+1}}^{-1} \Sigma_{t_{i+1},t_{i+2}} \cdots \Sigma_{t_{j-1},t_{j-1}}^{-1} \Sigma_{t_{j-1},t_j} = \Sigma_{t_i,t_j}.$ By the concatenation formula, $P^{t_0,\dots,t_m}$ and $P^{t_0,t_1} \circ \dots \circ P^{t_{m-1},t_m}$ are both centered Gaussian measures with identical covariance matrix and are thus equal.
\end{proof}

The following remark explains why a Markov Gaussian measure with continuous covariance function and positive variance has also a positive covariance function.

\begin{remq}\label{rq:pos_markov}
	Let $T$ be an interval and $P \in \PP\left(\R^T\right)$ a Gaussian measure with non-singular covariance function $K : T \times T \to \R$. In \cite[Theorem $1$]{mehr_certain_1965}, the authors proved that if $P$ satisfies the Markov property and $K$ is continuous, then there exists two functions $f,g$ such that, for all $s<t \in \R^2$, $K(s,t) =f(s)g(t)$. Since for all $t \in \R$, $K(t,t) = f(t)g(t) > 0$, both $f$ and $g$ do not vanish. As $f$ (resp. $g$) is continuous and does not vanish, $f$ (resp. $g$) is either positive or negative. Since for all $t \in T$, $f(t) g(t) >0$,  either $f$ and $g$ are both positive, or $f$ and $g$ are both negative. This implies $K$ is positive. Thus, every continuous covariance function of a Markov Gaussian measure with positive variance function is positive.
\end{remq}

\textbf{\large{Acknowledgment.}} I would  like to express my deep gratitude to Nicolas Juillet for introducing me to this research problem and for his constant support and valuable suggestions throughout the writing process. I would also like to extend my thanks to the University of Strasbourg for my stay during this period, with special appreciation to the probability team.

\vspace{0.5cm}
\textbf{\large{Copyright notice}:}
\begin{minipage}{1.8cm}
	\includegraphics[width=1.8cm]{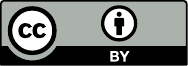}
\end{minipage} 
\vspace{0.1cm}This research was funded, in whole or in part, by the Agence nationale de la recherche (ANR), grant ANR-23-CE40-0017, and by the Austrian Science Fund (FWF), grant 10.55776/PIN1806324. A CC-BY public copyright license has been applied by the authors to the present document and will be applied to all subsequent versions up to the Author Accepted Manuscript arising from this submission, in accordance with the grants’ open access conditions.

\bibliographystyle{spr-ims-number}%
\bibliography{document.bib}%

\end{document}